\documentclass[preprint,12pt]{elsarticle}

\usepackage{amsthm}
\usepackage{lineno}
\usepackage{latexsym,amscd,amsfonts,enumerate,supertabular}
\usepackage{tabularx}
\usepackage{array}
\usepackage{bigstrut}
\usepackage{graphicx}
\usepackage{epsfig}
\usepackage{amssymb}
\usepackage[intlimits]{amsmath}
\usepackage{mathrsfs}
\usepackage{pst-plot}
\usepackage{bm}

\usepackage{ifpdf}
\usepackage{color}
\usepackage{multicol}
\usepackage{multirow}
\usepackage{hhline}
\usepackage{makecell}
\usepackage{xcolor}
\usepackage{float}
\usepackage{subfigure}
\usepackage{mathabx}

\usepackage{makecell,bigstrut}
\usepackage{pstricks,pst-node,pst-tree,tikz}
\usepackage{soul}

\tikzset{
treenode/.style = {inner sep=0.2pt, text centered, font=\rmfamily},
arn_b/.style = {treenode, circle, white, draw=black, fill=black, text width=1.06mm},
arn_w/.style = {treenode, circle, black, draw=black, text width=1.06mm},
w/.style = {treenode, circle, black, draw=black, text width=1.3mm},
textnode/.style = {text width=2mm},
bigcirc_w/.style = {draw, circle, black, draw=black, minimum size=2.2cm, text width=1.9mm}
}

\def\tdOne{\begin{tikzpicture}[-, grow=up, level 1/.style={sibling distance=5mm,level distance=4mm}]
\node [arn_b] {};
\end{tikzpicture}}

\def\tdTwoOne{\begin{tikzpicture}[-, grow=up, level 1/.style={sibling distance=5mm,level distance=4mm}]
\node [arn_b] {} child{ node [arn_w]{} };
\end{tikzpicture}}

\def\tdThreeOne{\begin{tikzpicture}[-, grow=up, level 1/.style={sibling distance=5mm,level distance=4mm}]
\node [arn_b] {} child{ node [arn_w]{} child{ node [arn_b] {}}};
\end{tikzpicture}}

\def\tdFourOne{\begin{tikzpicture}[-, grow=up, level 1/.style={sibling distance=5mm,level distance=4mm}]
\node [arn_b] {} child{ node [arn_w]{} child{node [arn_b]{} child{node [arn_w] {}}}};p
\end{tikzpicture}}

\def\tdFourTwo{\begin{tikzpicture}[-, grow=up, level 1/.style={sibling distance=5mm,level distance=4mm}]
\node [arn_b] {} child{ node [arn_w]{} child{ node [arn_b] {}} child{ node [arn_b] {}}};
\end{tikzpicture}}
\def\tdFiveOne{\begin{tikzpicture}[-, grow=up, level 1/.style={sibling distance=5mm,level distance=4mm}]
\node [arn_b] {} child{ node [arn_w]{} child{ node [arn_b]{} child{node [arn_w]{} child{node [arn_b] {}}}}};
\end{tikzpicture}}

\def\tdFiveTwo{\begin{tikzpicture}[-, grow=up, level 1/.style={sibling distance=5mm,level distance=4mm}]
\node [arn_b] {} child{ node [arn_w]{} child{ node [arn_b] {} child{ node [arn_w] {}}} child{ node [arn_b] {} }};
\end{tikzpicture}}

\def\tdFiveThree{\begin{tikzpicture}[-, grow=up, level 1/.style={sibling distance=5mm,level distance=4mm}]
\node [arn_b] {} child{ node [arn_w]{} child{ node [arn_b] {}} child{ node [arn_b] {}} child{ node [arn_b] {}}};
\end{tikzpicture}}

\def\tdSixOne{\begin{tikzpicture}[-, grow=up, level 1/.style={sibling distance=5mm,level distance=4mm}]
\node [arn_b] {} child{ node [arn_w]{} child{ node [arn_b]{} child{ node [arn_w]{} child{node [arn_b]{} child{node [arn_w] {}}}}}};
\end{tikzpicture}}

\def\tdSixTwo{\begin{tikzpicture}[-, grow=up, level 1/.style={sibling distance=5mm,level distance=4mm}]
\node [arn_b] {} child{ node [arn_w]{} child{node [arn_b]{} child{node [arn_w] {} child{ node [arn_b] {}} child{ node [arn_b] {}}}}};
\end{tikzpicture}}

\def\tdSixThree{\begin{tikzpicture}[-, grow=up, level 1/.style={sibling distance=5mm,level distance=4mm}]
\node [arn_b] {} child{ node [arn_w]{} child{ node [arn_b] {} child{ node [arn_w] {} child{ node [arn_b] {}}}}child{ node [arn_b] {}}};
\end{tikzpicture}}

\def\tdSixFour{\begin{tikzpicture}[-, grow=up, level 1/.style={sibling distance=5mm,level distance=4mm}]
\node [arn_b] {} child{ node [arn_w]{} child{ node [arn_b] {} child{ node [arn_w] {}}} child{ node [arn_b] {} child{ node [arn_w] {}}}};
\end{tikzpicture}}

\def\tdSixFive{\begin{tikzpicture}[-, grow=up, level 1/.style={sibling distance=5mm,level distance=4mm}]
\node [arn_b] {} child{ node [arn_w]{} child{ node [arn_b] {} child{ node [arn_w] {}}} child{ node [arn_b] {}} child{ node [arn_b] {} } };
\end{tikzpicture}}

\def\tdSixSix{\begin{tikzpicture}[-, grow=up, level 1/.style={sibling distance=5mm,level distance=4mm}]
\node [arn_b] {} child{ node [arn_w]{} child{ node [arn_b] {}} child{ node [arn_b] {}} child{ node [arn_b] {}} child{ node [arn_b] {}}};
\end{tikzpicture}}

\definecolor{darkmagenta}{rgb}{0.55, 0.0, 0.55}

\ifpdf
  \DeclareGraphicsExtensions{.pdf,.png,.jpg}
\else
  \DeclareGraphicsExtensions{.eps}
\fi

\numberwithin{equation}{section}
\newtheorem{mydef}{Definition}[section]
\newtheorem{mythm}{Theorem}[section]

\newtheorem{mylem}{Lemma}[section]
\newtheorem{myrem}{Remark}[section]
\newtheorem{mycor}{Corollary}[section]

\begin{document}
\begin{frontmatter}
\title{{\bf
Exponentially fitted two-derivative DIRK methods for oscillatory differential equations
}}
 \tnotetext[label2]{The first author was supported by 2019 EMS-Simons research visit (type A1) grant. The second author is partially supported by NSF grant DMS--2012022. The fourth author was partially supported by National Natural Science Foundation of China (No. 11171155, No. 11871268) and  Natural Science Foundation of Jiangsu Province, China (No. BK20171370).}

\author[ul]{Julius O. Ehigie\corref{cor1}}
 \ead{jehigie@unilag.edu.ng}

\author[msu]{Vu Thai Luan}
 \ead{luan@math.msstate.edu}

\author[ul]{Solomon A. Okunuga}
\ead{sokunuga@unilag.edu.ng}

\author[nau]{Xiong You}
\ead{youx@njau.edu.cn}

\cortext[cor1]{Corresponding author}
\address[ul]{Department of Mathematics, University of Lagos, 23401, Nigeria}

\address[msu]{Department of Mathematics and Statistics, Mississippi State University, Mississippi State, MS, 39762, USA}

\address[nau]{Department of Applied Mathematics, Nanjing Agricultural University, Nanjing 210095, China}

\begin{abstract}
\small
In this work, we construct and derive a new class of exponentially fitted two-derivative diagonally implicit Runge--Kutta (EFTDDIRK) methods for the numerical solution of differential equations with oscillatory solutions. First, a general format of so-called  modified two-derivative diagonally implicit Runge--Kutta methods (TDDIRK) is proposed. Their order conditions up to order six  are derived by introducing a set of bi-coloured rooted trees and deriving new elementary weights. Next, we build exponential fitting conditions in order for these modified TDDIRK methods to treat oscillatory solutions, leading to EFTDDIRK methods. In particular,  a family of 2-stage fourth-order, a fifth-order, and a 3-stage sixth-order EFTDDIRK schemes are derived. These can be considered as superconvergent methods. The stability and phase-lag analysis of the new methods are also investigated, leading to optimized fourth-order schemes, which turn out to be much more accurate and efficient than their non-optimized versions.
 Finally, we carry out numerical experiments on some oscillatory test problems. Our numerical results clearly demonstrate the accuracy and efficiency of  the newly derived methods when compared with existing trigonometically/exponentially fitted implicit Runge--Kutta methods and two-derivative Runge--Kutta methods of the same order in the literature.
\small
\end{abstract}

\end{frontmatter}

\section{Introduction}
Oscillatory differential equations have often been used to model  oscillatory phenomena in various fields of applied sciences such as celestial mechanics \cite{Arnold1989}, molecular dynamics \cite{GriebelKnapekZumbusch2007}, quantum chemistry \cite{RossiBudroniMarchettiniCutiettaRusticiLiveri2009} and regulatory genomics \cite{ElowitzLeibler2000}, to mention a few.
This class of differential equations can be formulated as initial value problems of the general form
\begin{equation}\label{IVP}
   y'(x)=f(y(x)), \quad y(x_0)=y_0.
\end{equation}
Simulating such an oscillatory system \eqref{IVP} is not an easy task since it usually involves periodic or oscillating solutions. An extensive discussion on numerical analysis of  \eqref{IVP} including oscillatory systems can be found in \cite{HairerLubichWanner2006}. In particular, classical methods such as explicit Runge--Kutta schemes have their limitations (e.g., the lack of stability) in capturing the right behavior of oscillatory solutions. This forces them to use tiny time steps and thereby is inefficient. In this regard, standard methods such as implicit Runge--Kutta methods \cite{HW96} or recent advanced methods such as exponential integrators (e.g., see \cite{HO10,LuanMichels2020,Luan2020,Luan2017,LuanOstermann2016,LuanOstermann2013})
are preferable  since they may offer A-stable property or much larger stability regions, meaning less restriction to stepsizes. Here, we focus on diagonally implicit Runge--Kutta (DIRK) methods (or sometimes referred to as semi-implicit Runge--Kutta methods)  \cite{KennedyCarpenter2016}, among others, which have shown to be very attractive \cite{Butcher2016} over fully implicit methods in terms of computational efficiency. They were also designed for differential systems with eigenvalues on the imaginary axis where the phase errors (dispersion) and the numerical damping (dissipation) of the free oscillations in the numerical solution are small, see \cite{Franco2003}.

Motivated by the work of \cite{KastlungerWanner1972} which utilizes higher order derivatives to derive superconvergent (implicit) Runge--Kutta schemes, a class of two-derivative Runge--Kutta (TDRK) methods for  solving \eqref{IVP} was proposed in \cite{ChanTsai2010,TsaiChanWang2014}. An advantage of these methods is that the number of algebraic order conditions is significantly reduced in comparison with the classical Runge--Kutta methods of the same order, thereby allowing the construction of high-order schemes with only a few stages (see also \cite{GoekenJohnson2000, WuXia2006,AkanbiOkunugaSofoluwa2012,WusuAkanbiOkunuga2013,SealGucluChristlieb2014,TuraciOzis2017}). Following this, implicit TDRK methods were derived in \cite{ChanWangTsai2012,AhmadSenuIbrahim2019}, general linear TDRK  methods were presented in \cite{AbdiBrasHojjati2014,AbdiHojjati2011b,AbdiHojjati2015,ButcherHojjati2005}, and recently TDRK methods with optimal phase properties were constructed in \cite{FangYou2014,KalogiratouMonovasilisSimos2019,Krivovichev2020}.
In the case if a good estimate of frequency is known in advance, one can further improve the numerical solution of these methods by incorporating the exponential fitting idea  \cite{Gautshci1961,Bettis1979,Ixaru1984,Coleman1989} in order to integrate exactly the system who solutions are linear combinations of ${\rm e}^{\pm{\rm i} \omega x}$ ($\omega$ is a prescribed frequency). This leads to  exponentially/trigonometrically fitted Runge-Kutta (EFRK) methods  whose coefficients are non-constant coefficients involving $\omega$ \cite{Simos1998,BergheMeyerVanDaeleHecke1999,BergheMeyerVanDaeleHecke2000}. Exponentially fitted symmetric and symplectic implicit Runge-Kutta method were derived in \cite{CalvoFrancoMontijanoRandez2008,CalvoFrancoMontijanoRandez2009,EhigieDiaoZhangFangHouYou2017}. An extensive survey of exponentially fitted methods can be found in \cite{Paternoster2012,WuYouWang2013} and the references therein. We also note that trigometrically/exponentially fitted two-derivative Runge-Kutta methods (EFTDRK) methods with optimal phase properties were derived in \cite{FangYouMing2014,ChenLiZhangYou2015,AhmadSenuIbrahimOthman2019}.

In this paper, we combine the idea of DIRK and EFTDRK methods to construct a new class of exponentially fitted two-derivative diagonally implicit Runge--Kutta (EFTDDIRK) methods up to order six. In contrast to the previous work, we directly formulate the general format of these methods specially for oscillatory systems (called modified two-derivative DIRK schemes) by allowing coefficients to be functions of a complex variable $\mu={\rm i}\omega h$ (${\rm i}^2=-1$) involving the frequency $\omega$  and the stepsize $h$. We then use a set of bi-coloured rooted trees to represent the expansion of the exact and the numerical solution, and thus derive new elementary weights. This allows us to easily derive the order conditions for high-order methods. Next, using the idea of EFTDRK approach, we obtain some preliminary results  leading to new exponential fitting conditions of high-order. With these in place, we were able to construct EFTDDIRK methods of high-order using only few stages. In particular, with $s=2$ stages we obtain superconvergent schemes of order $p=2s$ and $p=2s+1$, and with $s=3$ stages, a method of order $p=2s$ is derived. This is a significant improvement since previous  EFTDRK methods (either explicit or implicit) \cite{ChanTsai2010, KalogiratouMonovasilisSimos2019,AhmadSenuIbrahim2019,ChanWangTsai2012} only achieve orders $p=2s$ for  $s=2$ stages and $p=2s-1$ for  $s=3$ stages.

The outline of the paper is as follows. In Section~\ref{NumericalMethod}, we introduce the modified two-derivative diagonally implicit Runge--Kutta method (TDDIRK) for solving \eqref{IVP} and derive their order conditions (Lemma~\ref{OrderProposition} and Theorem~\ref{ClassicalOrderConditions}). In Section~\ref{ExponentialCondition}, we present a theoretical framework to establish  exponential fitting conditions for these modified TDDIRK methods (Theorem~\ref{theorem3.1}). This gives rise to EFTDDIRK methods. With the classical order conditions and exponential fitting conditions in hand, we are able to derive EFTDDIRK schemes of orders up to six in Section~\ref{Construct}. Our main results in this section are the new schemes  $\mathtt{EFTDDIRK2s4(c_1,c_2,\phi)}$ (a family of 2-stage fourth-order methods), $\mathtt{EFTDDIRK2s5}$ (a 2-stage fifth-order method), and $\mathtt{EFTDDIRK3s6}$ (a 3-stage   sixth-order method).
The stability and phase properties of these new methods are studied in Section~\ref{StabilityPhase}, in which the new fourth-order schemes are optimized for their accuracy. A technique for the estimation of frequency needed for the implementation of  EFTDDIRK methods is discussed in~Section \ref{FreqDetermine}. Section~\ref{experiments} devotes to the numerical experiments, in which  the effectiveness of the new EFTDDIRK methods is demonstrated. Finally, we give some concluding remarks in Section~\ref{conclu}.

\section{Numerical method}\label{NumericalMethod}

In this section,  we introduce so-called \emph{modified two-derivative diagonally implicit Runge--Kutta methods} (TDDIRK) for solving \eqref{IVP} and derive their order conditions. Our idea was motivated by \cite{ChenLiZhangYou2015}, in which the modified explicit two-derivative Runge--Kutta (TDRK) methods were introduced.
\subsection{Modified two-derivative diagonally implicit Runge--Kutta methods}\label{subsec2.1}

For the numerical solution of oscillatory problems \eqref{IVP}, we define an $s$-stage modified TDDIRK method  by the following formulation:
\begin{subequations}\label{modTDDIRK}
\begin{align}
Y_i &= y_n+\xi_i(\mu)c_{i} h f(y_n)+h^2\sum\limits_{j=1}^{i} a_{ij}(\mu)g(Y_j),\quad i=1,\ldots,s,\label{InternalStages}\\
 y_{n+1} &= y_n+h f(y_n)+h^2\sum\limits_{i=1}^s b_i(\mu)g(Y_i).\label{Updates}
\end{align}
\end{subequations}
Here, $h$ is the stepsize, $g(y)=y''=f'(y)f(y)$, and $\mu={\rm i}\omega h$, where $\omega>0$  is an accurate estimate of the
principal frequency of the problem.
The coefficients
$\xi_i(\mu), b_i(\mu)$, and $a_{ij}(\mu)$  are
assumed to be complex differentiable  on their domains. Furthermore,  they are also supposed to be even functions of
$\mu$ (this will be justified later in Section~\ref{ExponentialCondition}, see Lemma~\ref{lemma3.2}).

Note that the coefficients in \eqref{modTDDIRK} can be represented in a Butcher's type tableau with $\bm{c}=(c_1,\ldots,c_s)^T$, $\bm{\xi}(\mu)=(\xi_1(\mu),\ldots, \xi_s(\mu))^T$, $\bm{b}(\mu)=(b_1(\mu),\ldots,b_s(\mu))^T$, and $\bm{A}(\mu)=[a_{ij}(\mu)]_{i,j=1}^s$.

\subsection{Local error and order conditions}

To derive general order conditions for the proposed scheme \eqref{modTDDIRK}, we consider  $y_{n+1}$ as the numerical solution at $x_n +h$ obtained after one step starting from $y(x_n)$ with the local assumption that $y(x_n)=y_n$. Similar to the derivation of the order conditions for Runge--Kutta methods, the idea is then to compare the series expansion of  $y_{n+1}$ with that of the exact solution  $y(x_n+h)$  (in terms of the stepsize $h$) to study the local error
\begin{equation}\label{LTE:def}
    LTE_{n}=y(x_n+h)-y_{n+1}.
\end{equation}
For a sufficient smooth solution $y(x)$, the modified TDDIRK method ~\eqref{modTDDIRK} is said to be of order $p$ (consistency) if the local truncation error of the
solution satisfies
\begin{equation}\label{LTE}
    LTE_{n}=\mathcal{O}(h^{p+1}) \quad \text{as} \quad
    h\rightarrow 0.
\end{equation}
For deriving schemes of high order, it turns out that the expansion of the numerical and exact solution become much more complicated. In this regard, the rooted tree theory significantly simplifies the derivation of order conditions (e.g., see \cite{Butcher2016,HairerNorsettWanner1993} for classical Runge-Kutta methods and see \cite{ChenLiZhangYou2015} for modified TDRK methods). Here, we adapt this approach  to the proposed modified TDDIRK methods.

\subsubsection{Expansion of the exact solution}

Assume that $f(y)$ is sufficiently differentiable (thus $g(y)$), one can then expand $y(x_n+h)$ using the Taylor series about $x_n$ to obtain
\begin{equation}\label{ExpandExact}
y(x_n+h)=y_n+hf+\dfrac{h^2}{2!}g+\dfrac{h^3}{3!}g'f+\dfrac{h^4}{4!}\big(g'g+g''(f,f)\big)+\ldots
\end{equation}
where $f$ and $g$ and their derivatives are the short abbreviations representing that they are  evaluated at $y_n$. Generalizing  \eqref{ExpandExact} in terms of rooted trees, one can show that
\begin{equation}\label{TreeExact}
y(x_n+h)=y_n+\sum\limits_{\tau\in\mathcal{BT}}\dfrac{h^{\rho(\tau)}}{\rho(\tau)!}\alpha(\tau)\mathcal{F}(\tau)(y_n)
\end{equation}
where $\mathcal{BT}$ is the set of bi-coloured trees
$$
\mathcal{BT}=\{\tdOne, \tdTwoOne, \tdThreeOne, \tdFourOne, \tdFourTwo, \ldots \},
$$
associated with the recursively generated elementary differentials
$\mathcal{F}(\tau)(y_n)=\{f,g,g'f,g'g,g''(f,f),\ldots\}$, $\tau\in \mathcal{BT}$, $\rho(\tau)$ is the order of the tree $\tau$ (the number of vertices of $\tau$), and $\alpha(\tau)$ is the number of monotonic labellings of $\tau \in \mathcal{BT}$, as respectively defined in a recursive approach, see  \cite[Appendix]{ChenLiZhangYou2015}:
\[
\rho(\tau) =
\begin{cases}
    1, & \hbox{$\tau=\tdOne$} \\
    2, & \hbox{$\tau=\tdTwoOne$} \\
1+\sum_{i=1}^{m} \rho(\tau_i), & \tau=[\tau_1,\ldots, \tau_m ]_2
\end{cases}
\]
and
\[
\alpha(\tau) =
\begin{cases}
    1, & \hbox{$\tau=\tdOne$} \\
    2, & \hbox{$\tau=\tdTwoOne$} \\
(\rho(\tau) - 2)!  \prod_{i=1}^{k} \frac{1}{n_i !} \Big(\frac{\alpha (\tau_i) }{ \rho(\tau_i)}\Big)^{n_i}, & \tau=[\tau^{n_1}_1,\ldots, \tau^{n_k}_k ]_2.
\end{cases}
\]
Here, the fact that a tree $\tau = [\tau_1,\ldots, \tau_m]_2 \in  \mathcal{BT}$ does not depend on the ordering of its branches which themselves need not be distinct, we suppose it has only $k$ distinct branches among $\tau_1,\ldots, \tau_m $ (say $\tau_1,\ldots, \tau_k$) with their number of occurrences denoted by $n_1,\ldots, n_k$ ($n_1+\ldots +n_k = m$), respectively. Thus, we rewrite such a tree $\tau$ as
$
\tau=[\tau^{n_1}_1,\ldots, \tau^{n_k}_k ]_2.
$

\subsubsection{Expansion of the numerical solution}
First, inserting the internal stages \eqref{InternalStages} into \eqref{Updates}  gives
\begin{equation}\label{combinedModTDDIRK}
y_{n+1}=y_n+hf+h^2 \sum\limits_{i=1}^sb_i(\mu)g\big(y_n+\xi_i(\mu)c_i hf+h^2\sum\limits_{j=1}^{i}a_{ij}(\mu)g(Y_j)\big).
\end{equation}
Then, we expand the function $g$ in \eqref{combinedModTDDIRK} in  Taylor series about $y_n$ to obtain
\begin{equation}\label{ExpandNM}
\begin{aligned}
y_{n+1}=y_n&+ \dfrac{h}{1!}\cdot 1! \cdot f+\dfrac{h^2}{2!}\cdot 2! \cdot 1\cdot \sum\limits_{i=1}^sb_i(\mu)g+\dfrac{h^3}{3!}\cdot 3! \cdot 1 \cdot \sum\limits_{i=1}^s  b_i(\mu)\big(\xi_i(\mu)c_i\big) g'f\\
&+\dfrac{h^4}{4!} \cdot 4! \cdot 1 \cdot \sum\limits_{i=1}^s b_i(\mu)\sum\limits_{j=1}^i  a_{ij}(\mu)g'g\\
&+\dfrac{h^4}{4!} \cdot 12 \cdot 1 \cdot \sum\limits_{i=1}^s b_i(\mu) (\xi_i(\mu)c_i\big)^2g''(f,f)+\ldots
\end{aligned}
\end{equation}
Now using the set of bi-coloured trees $\mathcal{BT}$, we derive a general expansion of  \eqref{ExpandNM} as
\begin{equation}\label{TreeNumerical}
y_{n+1}=y_n+\sum\limits_{\tau\in\mathcal{BT}}\dfrac{h^{\rho(\tau)}}{\rho(\tau)!}\gamma(\tau)\alpha(\tau)\Phi(\tau)\mathcal{F}(\tau)(y_n),
\end{equation}
where
$\rho(\tau)$ and $\alpha(\tau)$ were defined above;
$\gamma(\tau)$ is the density of $\tau$, and  $\Phi(\tau)$ is the elementary weight function (depending on $\bm{A}(\mu)$, $\bm{b}(\mu)$, $\bm{\xi}(\mu)$ and $\bm{c}$), which are recursively defined respectively as follows:
$$
\gamma(\tau)=\left\{
  \begin{array}{ll}
    1, & \hbox{$\tau=\tdOne$} \\
    2, & \hbox{$\tau=\tdTwoOne$} \\
    \rho(\tau)(\rho(\tau)-1)\gamma(\tau_1)\cdots \gamma(\tau_m), & \hbox{$\tau=[\tau_1,\ldots,\tau_m]_{2}$,}
  \end{array}
\right.
$$
and
\begin{equation}\label{eq:Phi}
\Phi(\tau)=\left\{
  \begin{array}{ll}
    \sum\limits_{i=1}^sb_i(\mu), & \hbox{$\tau=\tdTwoOne$} \\
    \sum\limits_{i=1}^sb_i(\mu)\Phi_i(\tau_1)\cdots\Phi_i(\tau_m) & \hbox{$\tau=[\tau_1,\ldots,\tau_m]_{2}$}
  \end{array}
\right.
\end{equation}
with
\begin{equation}\label{eq:Phi:sub}
\Phi_i(\tau)=\left\{
  \begin{array}{ll}
    \xi_i(\mu)c_i, & \hbox{$\tau=\tdOne$} \\[2ex]
    \sum\limits_{j=1}^i a_{ij}(\mu), & \hbox{$\tau=\tdTwoOne$} \\
    \sum\limits_{j=1}^i a_{ij}(\mu)\Phi_j(\tau_1)\cdots\Phi_j(\tau_m), & \hbox{$\tau=[\tau_1,\ldots,\tau_m]_{2}$.}
  \end{array}
\right.
\end{equation}
The elementary weight function $\Phi(\tau)$ defined in \eqref{eq:Phi}--\eqref{eq:Phi:sub} is new and applies to our modified TDDIRK method ~\eqref{modTDDIRK}.  Consider a special case of our method, e.g., when $ \xi_i(\mu) = 1$ and $a_{ij}(\mu) = a_{ij}$ (constant coefficients), we note that using this elementary weight function gives the similar results to the elementary weight function  given in \cite[Appendix]{ChenLiZhangYou2015},  which was defined in a much more complicated way.

\subsubsection{Local error and derivation of the order conditions}

By subtracting the numerical solution \eqref{TreeNumerical} from the exact solution \eqref{TreeExact}, the local truncation error \eqref{LTE:def} takes the form
\begin{equation}\label{LTEGen}
LTE=\sum\limits_{\tau\in\mathcal{BT}}\dfrac{h^{\rho(\tau)}}{\rho(\tau)!}\big( \gamma(\tau)\Phi(\tau) - 1 \big)\alpha(\tau)\mathcal{F}(\tau)(y_n).
\end{equation}
From this, one obtains the following result at once.
\begin{mylem}\label{OrderProposition}
Assuming that the solution of the initial value problem  \eqref{IVP} is sufficient smooth. The modified two-derivative diagonally implicit Runge-Kutta method \eqref{modTDDIRK} has order of consistency $p$ if
\begin{equation}\label{OrderTheorem}
\Phi(\tau)=\dfrac{1}{\gamma(\tau)}+\mathcal{O}\left(h^{p+1-\rho(\tau)} \right) \ \text{for all} \quad \tau \in \mathcal{BT} \ \text{with} \quad \rho(\tau)\leq p.
\end{equation}
\end{mylem}
Using \eqref{OrderTheorem}, we obtain in Table \ref{Rootedtrees6} the general order conditions (in vector form) for  the modified TDDIRK methods \eqref{modTDDIRK} of order $p$  ($2\leq p\leq 6$), corresponding to the bi-coloured trees, the order trees, and their elementary differentials
(note that, there $\bm{e}=[1,\ldots,1]^T$ and  ``*" denotes the component-wise product).
\setlength{\extrarowheight}{1.5 pt}
\renewcommand{\arraystretch}{1.35}
\begin{table}
\centering
\caption{Elementary differentials and general order conditions for the modified TDDIRK methods \eqref{modTDDIRK} for all  trees $\tau \in \mathcal{BT} \ \text{with} \ 2 \leq \rho(\tau)\leq 6$. }
\label{Rootedtrees6}
\small
\vspace{1.2mm}
\resizebox{\textwidth}{!}{%
\begin{tabular}{|c|c|c|c|c|}
\hline
Tree ($\tau$)& $\rho(\tau)$ &$\mathcal{F}(\tau)(y_n)$ & Order condition(s)& No.\\[2ex]
\hline
\raisebox{-2mm}{\tdTwoOne}&2&$g$&$\bm{b}(\mu)\cdot \bm{e}=\tfrac{1}{2}+\mathcal{O}\big(h^{p-1}\big)$&1\\[2ex]
\hline
\raisebox{-3mm}{\tdThreeOne}&3&$g'f$&$\bm{b}(\mu)\cdot \big(\bm{\xi}(\mu) \text{*} \bm{c}\big)=\tfrac{1}{6}+\mathcal{O}\big(h^{p-2}\big)$&2\\[2ex]
\hline
\raisebox{-3mm}{\tdFourOne}&4&$g'g$&$\bm{b}(\mu)\cdot \big( \bm{A}(\mu)\bm{e}\big)=\tfrac{1}{24}+\mathcal{O}\big(h^{p-3}\big)$&3\\[2ex]
\raisebox{-3mm}{\tdFourTwo}&4&$g''(f,f)$&$\bm{b}(\mu)\cdot \big( \bm{\xi}(\mu) \text{*} \bm{c}\big)^2=\tfrac{1}{12}+\mathcal{O}\big(h^{p-3}\big)$&4\\[2ex]
\hline
\raisebox{-5mm}{\tdFiveOne}&5&$g'g'f$&$\bm{b}(\mu)\cdot \big( \bm{A}(\mu)(\bm{\xi}(\mu)\text{*}\bm{c})\big)=\tfrac{1}{120}+\mathcal{O}\big(h^{p-4}\big)$&5\\[2ex]
\raisebox{-3mm}{\tdFiveTwo}&5&$g''(f,g)$&$\bm{b}(\mu) \cdot \big((\bm{\xi}(\mu) \text{*} \bm{c}) \text{*} (\bm{A}(\mu)\bm{e})\big)=\tfrac{1}{40}+\mathcal{O}\big(h^{p-4}\big)$&6\\[2ex]
\raisebox{-3mm}{\tdFiveThree}&5&$g'''(f,f,f)$&$\bm{b}(\mu)\cdot \big( \bm{\xi}(\mu) \text{*} \bm{c}\big)^3=\tfrac{1}{20}+\mathcal{O}\big(h^{p-4}\big)$&7\\[2ex]
\hline
\raisebox{-5mm}{\tdSixOne}&6&$g'g'g$&$\bm{b}(\mu)\cdot \big( \bm{A}^2(\mu)\bm{e}\big)=\tfrac{1}{720}+\mathcal{O}\big(h^{p-5}\big)$&8\\[2ex]
\raisebox{-5mm}{\tdSixTwo}&6& $g'g''(f,f)$&$\bm{b}(\mu)\cdot \big( \bm{A}(\mu)(\bm{\xi}(\mu) \text{*} \bm{c})^2\big)=\tfrac{1}{360}+\mathcal{O}\big(h^{p-5}\big)$&9\\[2ex]
\raisebox{-5mm}{\tdSixThree} &6& $g''(f,g'f)$&$\bm{b}(\mu) \cdot \big(\big(\bm{\xi}(\mu) \text{*} \bm{c}\big) \text{*} \bm{A}(\mu)(\bm{\xi}(\mu) \text{*} \bm{c})\big)=\tfrac{1}{180}+\mathcal{O}\big(h^{p-5}\big)$&10\\[2ex]
\raisebox{-5mm}{\tdSixFour} &6& $g''(g,g)$&$\bm{b}(\mu)\cdot \big( \bm{A}(\mu)\bm{e}\big)^2=\tfrac{1}{120}+\mathcal{O}\big(h^{p-5}\big)$&11\\[2ex]
\raisebox{-5mm}{\tdSixFive} &6& $g'''(f,f,g)$&$\bm{b}(\mu) \cdot \big(\big(\bm{\xi}(\mu) \text{*} \bm{c}\big)^2 \text{*} (\bm{A}(\mu)\bm{e})\big)=\tfrac{1}{60}+\mathcal{O}\big(h^{p-5}\big)$&12\\[2ex]
\raisebox{-3mm}{\tdSixSix} &6& $g''''(f,f,f,f)$&$\bm{b}(\mu)\cdot \big(\bm{\xi}(\mu) \text{*} \bm{c}\big)^4=\tfrac{1}{30}+\mathcal{O}\big(h^{p-5}\big)$&13\\[2ex]
\hline
\end{tabular}
}
\end{table}

On the other hand, with the assumption that the coefficients $a_{ij}(\mu)$, $b_i(\mu)$ and $\xi_i(\mu)$ of \eqref{modTDDIRK} are sufficiently differentiable and even functions of $\mu={\rm i} \omega h$, they can be expanded as
\begin{subequations}\label{Coeffh}
\begin{align}
a_{ij}(\mu) &
= a_{ij}^{(0)} + a_{ij}^{(2)}\mu^2 + a_{ij}^{(4)}\mu^4  + \ldots 
= a_{ij}^{(0)}-a_{ij}^{(2)}\omega^2h^2+a_{ij}^{(4)}\omega^4h^4-\ldots, \label{Coeffh:a}   \\
b_i(\mu)&
= b_i^{(0)} + b_i^{(2)}\mu^2 + b_i^{(4)}\mu^4 + \ldots
= b_i^{(0)}-b_i^{(2)}\omega^2h^2+b_i^{(4)}\omega^4h^4-\ldots,  \label{Coeffh:b} \\
  \xi_i(\mu)&
  = \xi_i^{(0)} + \xi_i^{(2)} \mu^2 + \xi_i^{(4)}\mu^4  + \ldots
  = \xi_i^{(0)}-\xi_i^{(2)}  \omega^2 h^2+\xi_i^{(4)}\omega^4h^4-\ldots. \label{Coeffh:xi}
\end{align}
\end{subequations}
(since ${\rm i}^2=-1$).
Here, we only consider $\omega \neq 0$ since if otherwise ($\omega = 0$, i.e., $\mu = 0$) one has  $a_{ij}(\mu) = a_{ij}^{(0)}$, $b_i(\mu)= b_i^{(0)}$, $\xi_i(\mu) = \xi_i^{(0)}$ and thus \eqref{modTDDIRK}  is reduced to classical two-derivative Runge--Kutta methods.
With the scalar coefficients  from \eqref{Coeffh}, we define
\begin{equation}\label{CoeffcientsSigma}
\begin{array}{ll}
\bm{A}^{(\sigma)}=\left[
                    \begin{array}{cccc}
                      a_{11}^{(\sigma)} &  &  &  \\
                      a_{21}^{(\sigma)} & a_{22}^{(\sigma)} &  & \\
                      \vdots&\vdots& \ddots& \\
                      a_{s1}^{(\sigma)} & a_{s2}^{(\sigma)} & \ldots & a_{ss}^{(\sigma)}
                    \end{array}
                  \right],& \begin{array}{c}
\bm{b}^{(\sigma)}=[b_1^{(\sigma)}, b_2^{(\sigma)},\ldots,b_s^{(\sigma)}]^T,\\[3ex]
\bm{\xi}^{(\sigma)}=[\xi_1^{(\sigma)}, \xi_2^{(\sigma)},\ldots,\xi_s^{(\sigma)}]^T,
\end{array}
\end{array}
\end{equation}
(here $\sigma=0,2,4,\ldots$) to derive the classical order conditions for  \eqref{modTDDIRK} in the Theorem~\ref{ClassicalOrderConditions} below.
\begin{mythm}\emph{(Classical order conditions)} \label{ClassicalOrderConditions}
Under the conventional simplifying assumptions that
\begin{equation}\label{rowassumptions}
      \bm{\xi}^{(0)}=\bm{e}, \qquad \bm{A}^{(0)}\bm{e}=\dfrac{\bm{c}^2}{2},
\end{equation}
the scheme \eqref{modTDDIRK}  is of order  $p$ consistency (for $2\leq p\leq6$), if the following classical order conditions are fulfilled:
\begin{itemize}
  \item Order 2 requires
\begin{equation}\label{d1}
  \bm{b}^{(0)}\cdot \bm{e}=\dfrac{1}{2}.
\end{equation}
  \item Order 3 requires, in addition
\begin{equation}\label{d2}
  \bm{b}^{(0)}\cdot \bm{c}=\dfrac{1}{6}.
\end{equation}
  \item Order 4 requires, in addition
\begin{subequations}\label{OrderFour}
\begin{align}
  &\bm{b}^{(0)}\cdot \bm{c}^2=\dfrac{1}{12}, \label{d3}\\[2ex]
  &\bm{b}^{(2)}\cdot \bm{e}=0. \label{d4}
\end{align}
\end{subequations}
  \item Order 5 requires, in addition
\begin{subequations}\label{OrderFive}
\begin{align}
  &\bm{b}^{(0)}\cdot \big(\bm{A}^{(0)}\bm{c}\big)=\dfrac{1}{120}, \label{d5}\\[2ex]
  &\bm{b}^{(0)}\cdot \bm{c}^3=\dfrac{1}{20}, \label{d6}\\[2ex]
  &\bm{b}^{(0)} \cdot (\bm{\xi}^{(2)} \text{*} \bm{c})+ \bm{b}^{(2)} \cdot \bm{c}=0. \label{d7}
\end{align}
\end{subequations}
  \item Order 6 requires, in addition
\begin{subequations}\label{OrderSix}
\begin{align}
  &\bm{b}^{(0)}\cdot \big(\bm{A}^{(0)}\bm{c}^2\big)=\dfrac{1}{360}, \label{d8}\\[2ex]
  &\bm{b}^{(0)} \cdot \big(\bm{c} \text{*} (\bm{A}^{(0)}\bm{c})\big)=\dfrac{1}{180}, \label{d9}\\[2ex]
  &\bm{b}^{(0)}\cdot \bm{c}^4=\dfrac{1}{30}, \label{d10}\\[2ex]
   &\bm{b}^{(4)}\cdot \bm{e}=0, \label{d11}\\[2ex]
   &2\bm{b}^{(0)} \cdot \bm{A}^{(2)}\bm{e}+\bm{b}^{(2)}\cdot \bm{c}^2=0, \label{d12}\\[2ex]
   &2\bm{b}^{(0)} \cdot \big( \bm{\xi}^{(2)}\text{*}\bm{c}^2\big)+\bm{b}^{(2)}\cdot \bm{c}^2=0. \label{d13}
\end{align}
\end{subequations}
\end{itemize}
\end{mythm}
\begin{proof}
The main idea is to insert \eqref{Coeffh} into conditions 1--13 of Table \ref{Rootedtrees6} to derive the classical order conditions \eqref{d1}--\eqref{OrderSix}.
We now illustrate this procedure  for  the case $p=6$. First, consider the tree of order two $\tau=\tdTwoOne$. Inserting \eqref{Coeffh:b} into condition 1 of Table \ref{Rootedtrees6} yields
\begin{equation}\label{Rem:SubEqTwo}
  \big(\bm{b}^{(0)} \cdot \bm{e} \big)-\big(\bm{b}^{(2)} \cdot \bm{e} \big)\omega^2h^2+\big(\bm{b}^{(4)} \cdot \bm{e} \big)\omega^4h^4+\mathcal{O}\big(h^6 \big)=\dfrac{1}{2}+\mathcal{O}\big(h^5 \big).
\end{equation}
Taking the limit of both sides  of \eqref{Rem:SubEqTwo} as  $h$ approaches zero results in order condition \eqref{d1}. Employing \eqref{d1} and dividing both sides of \eqref{Rem:SubEqTwo} by $h^2$ gives
\begin{equation}\label{Rem:SubEqTwo1}
  -\big(\bm{b}^{(2)} \cdot \bm{e} \big)\omega^2+\big(\bm{b}^{(4)} \cdot \bm{e} \big)\omega^4h^2+\mathcal{O}\big(h^4 \big)=\mathcal{O}\big(h^3 \big).
\end{equation}
Again, taking the limit of both sides as $h$ approaches zero  ($\omega \neq 0$) implies order condition \eqref{d4}.
Dividing both sides of \eqref{Rem:SubEqTwo1} by $h^2$ and keeping in mind the use of \eqref{d4} gives
\begin{equation}\label{Rem:SubEqTwo2}
  \big(\bm{b}^{(4)} \cdot \bm{e} \big)\omega^4+\mathcal{O}\big(h^2 \big)=\mathcal{O}\big(h \big),
\end{equation}
which shows \eqref{d12} when taking the limit of both sides as $h\rightarrow 0$.

Next, consider $\tau=\tdThreeOne$. Inserting \eqref{Coeffh:b} and \eqref{Coeffh:xi} into  condition 2 of Table \ref{Rootedtrees6}  yields
\begin{equation}\label{Rem:SubEqThree1}
(\bm{b}^{(0)} \cdot \bm{c})-\big(\bm{b}^{(0)} \cdot (\bm{\xi}^{(2)} \text{*} \bm{c}) \big)\omega^2h^2-\big(\bm{b}^{(2)} \cdot \bm{c}\big)\omega^2h^2+\mathcal{O}\big(h^4 \big)=\dfrac{1}{6}+\mathcal{O}\big(h^4 \big).
\end{equation}
From \eqref{Rem:SubEqThree1}, as $h\rightarrow 0$, one immediately obtains order condition \eqref{d2}. With this, dividing \eqref{Rem:SubEqThree1} by $h^2$, we obtain
\begin{equation}\label{Rem:SubEqThree2}
-\big(\bm{b}^{(0)} \cdot (\bm{\xi}^{(2)} \text{*} \bm{c}) \big)\omega^2-\big(\bm{b}^{(2)} \cdot \bm{c}\big)\omega^2+\mathcal{O}\big(h^2 \big)=\mathcal{O}\big(h^2 \big).
\end{equation}
Therefore, as $h\rightarrow 0$, \eqref{d7} is confirmed.

Similarly, consider $\tau=\tdFourOne$. Inserting \eqref{Coeffh:a} and \eqref{Coeffh:b} into  condition 3 of Table \ref{Rootedtrees6} gives
\begin{equation}\label{Rem:SubEqFourOne1}
\bm{b}^{(0)}\cdot \bm{A}^{(0)}\bm{e}-\big(\bm{b}^{(0)} \cdot \bm{A}^{(2)}\bm{e}+ \bm{b}^{(2)} \cdot \bm{A}^{(0)}\bm{e}\big)\omega^2h^2+\mathcal{O}\big(h^4 \big)=\dfrac{1}{24}+\mathcal{O}\big(h^3 \big).
\end{equation}
Using \eqref{rowassumptions}, one obtains \eqref{d3} as  $h\rightarrow 0$. With \eqref{d3}, dividing both sides by $h^2$, we obtain
\begin{equation}\label{Rem:SubEqFourOne2}
-\big(2\bm{b}^{(0)} \cdot \bm{A}^{(2)}\bm{e}+ \bm{b}^{(2)} \cdot \bm{c}^2\big)\omega^2+\mathcal{O}\big(h^2 \big)=\mathcal{O}\big(h \big).
\end{equation}
Therefore,  as $h\rightarrow 0$, we obtain \eqref{d12}.

Furthermore, consider $\tau=\tdFourTwo$. Inserting \eqref{Coeffh:xi} and \eqref{Coeffh:b} into  condition 4 of Table~\ref{Rootedtrees6}  and using the first assumption in \eqref{rowassumptions} yields
\begin{equation}\label{Rem:SubEqFourTwo1}
\bm{b}^{(0)}\cdot  \bm{c}^2-\big(2\bm{b}^{(0)} \cdot (\bm{\xi}^{(2)}\text{*}\bm{c}^2)+ \bm{b}^{(2)} \cdot \bm{c}^2\big)\omega^2h^2+\mathcal{O}\big(h^4 \big)=\dfrac{1}{12}+\mathcal{O}\big(h^3 \big).
\end{equation}
Using \eqref{d3} and dividing both sides by $h^2$, we obtain
\begin{equation}\label{Rem:SubEqFourTwo2}
-\big(2\bm{b}^{(0)} \cdot (\bm{\xi}^{(2)}\text{*}\bm{c}^2)+ \bm{b}^{(2)} \cdot \bm{c}^2\big)\omega^2+\mathcal{O}\big(h^2 \big)=\mathcal{O}\big(h \big),
\end{equation}
which shows  \eqref{d13} as $h\rightarrow 0$.

Continuing in this manner for other trees $\tau$ in Table~\ref{Rootedtrees6}, the remaining classical order conditions in \eqref{d1}-\eqref{OrderSix} can be obtained.
\end{proof}


\section{Exponential fitting conditions}\label{ExponentialCondition}
As our focus is on the IVP problem  \eqref{IVP}  with oscillatory solutions, it is desirable to require that the proposed scheme can integrate exactly the system whose solutions involve a linear combinations of reference sets of the form
$\{ p_k (x){\rm e}^{\lambda_k x}\}$, where  $p_k (x)$ are polynomials in $x$  and $\lambda_k$ can be real or complex numbers (note that in the case
$\lambda_k= \pm{\rm i} \omega$, such a reference set becomes
$\{ p_k (x){\rm e}^{\pm{\rm i} \omega x}=p_k (x)(\cos( \omega x) \pm {\rm i} \sin(\omega x) \}$).
This enforces additional requirements on the coefficients of the two-derivative Runge--Kutta methods, which are referred to the so-called \emph{exponential fitting conditions}, see \cite{IxaruBerghe2004}.
Therefore, along with the required order conditions given in Table~1, in this section we will derive such conditions for  \eqref{modTDDIRK}.

Following  \cite{IxaruBerghe2004} (see also \cite{Paternoster2012} and references therein), for which explicit exponential-fitted Runge--Kutta/Runge-Kutta-Nystr\"{o}m methods were
presented, the idea is to determine fitting conditions  so that the local  truncation errors of the internal stages $Y_i$ and the final stage $y_{n+1}$  of \eqref{modTDDIRK} vanish on a certain reference set.
For this, we introduce the following linear operators
 \begin{subequations}\label{optLiL}
 \begin{align}
\mathcal{L}_i[h, \bm{A}, \bm{c}, \bm{\xi}]y(x)&=y(x+c_ih)-y(x)-\xi_i(\mu)c_ihy'(x)-h^2\sum\limits_{j=1}^ia_{ij}(\mu)y''(x+c_jh), \label{optLiL:a}\\
\mathcal{L}[h, \bm{b}, \bm{c}]y(x)&=y(x+h)-y(x)-hy'(x)-h^2\sum\limits_{i=1}^sb_i(\mu)y''(x+c_ih), \label{optLiL:b}
\end{align}
\end{subequations}
which are associated with the internal stages and the final stage, respectively ($i=1,\ldots,s$).
\begin{mydef}
The scheme  \eqref{modTDDIRK}  is said to be  exponentially fitted TDDIRK  (EFTDDIRK) method of degree $(K,L)$ if
\begin{equation}\label{zeroL}
\mathcal{L}[h, \bm{b}, \bm{c}]y(x) \equiv 0\quad{\rm and}\quad \mathcal{L}_i[h, \bm{A}, \bm{c}, \bm{\xi}]y(x)\equiv 0,\quad i=1,\ldots,s
\end{equation}
for all $y(x)$ that belong to the subspace
\begin{equation}\label{Fae}
\mathcal{F}_{K,L}=\emph{span}\{x^k{\rm e}^{ \lambda_l x}, \ \lambda_l \in \mathbb{C},  has  \ k=0,\ldots,K, \ l=1,\ldots,L\}.
\end{equation}
\end{mydef}
On  $\mathcal{F}_{K,L}$, we prove the following important properties of the operators $ \mathcal{L}_i$ and $ \mathcal{L}$  defined in \eqref{optLiL}.
\begin{mylem}\label{lemma3.1}
For all  $y_K (x)=x^K{\rm e}^{ \lambda_l x} \in \mathcal{F}_{K,L}$  and $\lambda_l \in \mathbb{C}$, we have
 \begin{subequations}\label{propertyLiL}
 \begin{align}
 \mathcal{L}_i[h, \bm{A}, \bm{c}, \bm{\xi}]y_K (x)&= \sum_{m=0}^{K}  \binom{K}{m} y_m (x) \mathcal{L}_i[h, \bm{A}, \bm{c}, \bm{\xi}]y_{K-m} (0),  \label{propertyLiL:a}  \\
 \mathcal{L}[h, \bm{b}, \bm{c}]y_K (x)&= \sum_{m=0}^{K}  \binom{K}{m} y_m (x)  \mathcal{L}[h, \bm{b}, \bm{c}] y_{K-m} (0). \label{propertyLiL:b}
 \end{align}
\end{subequations}
\end{mylem}
\begin{proof}
With $y_K (x)=x^K{\rm e}^{ \lambda_l x} $, one can verify that
\begin{subequations} \label{eq:yK}
\begin{align}
y_K (x+c_i h) &=(x+c_i h)^K{\rm e}^{ \lambda_l (x+ c_i h)}=  \sum_{m=0}^{K}  \binom{K}{m} y_m (x) y_{K-m}(c_i h),  \label{eq:yKa} \\
y'_K (x) & =K y_{K-1} (x) + \lambda_l  y_K (x), \label{eq:yKb}  \\
y''_K (x) & =K (K-1) y_{K-2} (x) +2K \lambda_l  y_{K-1} (x) +  \lambda^2_l  y_K (x). \label{eq:yKc}
\end{align}
\end{subequations}
Inserting  \eqref{eq:yKa} and  \eqref{eq:yKb} into \eqref{optLiL:a} gives
\begin{equation} \label{eq:LiyK}
\begin{aligned}
 \mathcal{L}_i[h, & \bm{A}, \bm{c},   \bm{\xi}]y_K (x) = \sum_{m=0}^{K}  \binom{K}{m} y_m (x) y_{K-m}(c_i h) - y_K (x)\\
& -\gamma_i (\mu) c_i h [K y_{K-1} (x) + \lambda_l  y_K (x)]   -h^2\sum\limits_{j=1}^ia_{ij}(\mu) y''_K (x+ c_j h).
 \end{aligned}
\end{equation}
Using  \eqref{eq:yKc} and  \eqref{eq:yKa} (with $c_j$ in place of $c_i$), we obtain
\begin{equation} \label{eq:yKdouble}
\begin{aligned}
 y''_K (x+ c_j h) &=K(K-1)  y_{K-2} (x+c_j h) +    2K \lambda_l  y_{K-1} (x+ c_j h) +  &  \lambda^2_l  y_K (x+c_j h) \\
 &=K(K-1) \sum_{m=0}^{K-2}  \binom{K-2}{m} y_m (x) y_{K-2-m}(c_j h)\\
 & + 2K \lambda_l  \sum_{m=0}^{K-1}  \binom{K-1}{m} y_m (x) y_{K-1-m}(c_j h)\\
 & + \lambda^2_l  \sum_{m=0}^{K}  \binom{K}{m} y_m (x) y_{K-m}(c_j h).
 \end{aligned}
\end{equation}
Next, we insert \eqref{eq:yKdouble} into \eqref{eq:LiyK} and factor out the  terms \ $y_K (x),  \binom{K}{K-1}  y_{K-1}(x)=K y_{K-1}(x),$ and $ \sum_{m=0}^{K-2}  \binom{K}{m} y_m (x)$   from the obtained result. It is then not difficult to show that \eqref{eq:LiyK} becomes
\begin{equation*}
\begin{aligned}
\mathcal{L}_i[h, \bm{A}, \bm{c},   \bm{\xi}]y_K (x) & =y_K (x)   \mathcal{L}_i[h, \bm{A}, \bm{c}, \bm{\xi}]y_{0} (0)
+ K y_{K-1} (x)  \mathcal{L}_i[h, \bm{A}, \bm{c}, \bm{\xi}]y_{1} (0)  \\
& + \sum_{m=0}^{K-2}  \binom{K}{m} y_m (x)  \mathcal{L}_i[h, \bm{A}, \bm{c}, \bm{\xi}]y_{K-m} (0)
 \end{aligned}
\end{equation*}
which proves \eqref{propertyLiL:a}. \\
Note that using  \eqref{eq:yKa} (with $c_i=1$) and \eqref{eq:yKdouble} (with $c_i$ in place of $c_j$), the proof of \eqref{propertyLiL:b} can be carried out in the same way. We omit the details.
\end{proof}
Using the result of Lemma~\ref{lemma3.1}, we are ready to state the fitting conditions for the proposed method \eqref{modTDDIRK}.

\begin{mythm}
\label{theorem3.1}
\emph{(Fitting conditions)}
Under the following conditions
 \begin{equation}\label{eq:fitcondition}
  \mathcal{L}_i[h, \bm{A}, \bm{c}, \bm{\xi}]y_{m} (0)=0, \quad
  \mathcal{L}[h, \bm{b}, \bm{c}] y_{m} (0)=0,
\end{equation}
for all $y_m (x)=x^m{\rm e}^{ \lambda_l x} \in \mathcal{F}_{K,L}$ with
$\lambda_l \in \mathbb{C}$,
$m=0, 1, 2, \ldots, K$ and $l=0, 1,2, \ldots, L$, the scheme \eqref{modTDDIRK} is an EFTDDIRK method of degree  $(K,L)$.
In particular,  we have:
\begin{itemize}
\item Degree $(0, L)$ requires the following fitting conditions
\begin{subequations} \label{eq:K0}
\begin{align}
 &  (\lambda_l h)^2 \sum_{j=1}^{i} a_{ij} (\mu) {\rm e}^{c_j  \lambda_l h}+ \xi_i (\mu) c_i  \lambda_l h= {\rm e}^{c_i  \lambda_l h} -1,  \label{eq:K0a}\\
 & (\lambda_l h)^2 \sum_{i=1}^{s} b_i (\mu) {\rm e}^{c_i \lambda_l h}  =  {\rm e}^{\lambda_l h} -1 - \lambda_l h \label{eq:K0b}.
 \end{align}
\end{subequations}
\item Degree $(1, L)$ requires, in addition to  \eqref{eq:K0}, the following conditions
\begin{subequations} \label{eq:K1}
\begin{align}
  & \sum_{j=1}^{i} a_{ij} (\mu) [2 \lambda_l h + c_j  (\lambda_l h)^2] {\rm e}^{c_j  \lambda_l h}+ \xi_i (\mu) c_i = c_i {\rm e}^{c_i  \lambda_l h},   \label{eq:K1a}\\
 & \sum_{i=1}^{s} b_i (\mu) [2 \lambda_l h + c_i  (\lambda_l h)^2] {\rm e}^{c_i  \lambda_l h}  =  {\rm e}^{ \lambda_l h} -1.  \label{eq:K1b}
 \end{align}
\end{subequations}
\item Degree $(K \ge 2, L)$ requires, in addition to \eqref{eq:K0} and \eqref{eq:K1}
\begin{subequations} \label{eq:K2}
\begin{align}
 & \sum_{j=1}^{i} a_{ij} (\mu) \big[K(K-1)c^{K-2}_j + 2 K c^{K-1}_j \lambda_l h + c^{K}_j  (\lambda_l h)^2  \big] {\rm e}^{c_j \lambda_l h } = c^{K}_i  {\rm e}^{c_i \lambda_l h },   \label{eq:K2a}\\
 & \sum_{i=1}^{s} b_i (\mu)  \big[K(K-1)c^{K-2}_i + 2 K c^{K-1}_i \lambda_l h + c^{K}_i (\lambda_l h)^2  \big] {\rm e}^{c_i \lambda_l h } = {\rm e}^{\lambda_l h }.  \label{eq:K2b}
 \end{align}
 \end{subequations}
\end{itemize}
 \end{mythm}
 \begin{proof}
The first part of Theorem~\ref{theorem3.1} follows directly from Lemma~\ref{lemma3.1}. More specifically,  \eqref{eq:fitcondition}  implies at once \eqref{zeroL} for all $y(x) \in  \mathcal{F}_{K,L}$.
Next, we work out  \eqref{eq:fitcondition},  i.e.,
\begin{subequations} \label{eq:LiLm0}
\begin{align}
& \mathcal{L}_i[h,  \bm{A}, \bm{c},   \bm{\xi}]y_m (0) = y_m (c_i h) - y_m (0) -\xi_i (\mu) c_i h  y'_m (0)  -h^2\sum\limits_{j=1}^ia_{ij}(\mu) y''_m (c_j h)=0,  \label{eq:LiLm0:a}\\
&  \mathcal{L}[h,  \bm{b}, \bm{c}]y_m (0)  = y_m ( h) - y_m (0) - h y'_m (0)  -h^2\sum\limits_{i=1}^{s} b_i (\mu) y''_m (c_i h)=0. \label{eq:LiLm0:b}
\end{align}
\end{subequations}
For $K=0$,  we have  $m=0$ and thus consider $y_0 (x)= {\rm e}^{\lambda_l x}$. A simple calculation shows that $y_0 (c_i h)= {\rm e}^{c_i \lambda_l h}, \ y_0 (0)=1, \ y'_0 (0)=\lambda_l, $ and $y''_0 (c_j h)=\lambda^2_l  e^{ c_j \lambda_l h }$. Inserting these relations into \eqref{eq:LiLm0}, we immediately get \eqref{eq:K0}.
For $K=1$, we have $m=0, 1$, and thus consider, in addition to $y_0 (x)$, $y_1 (x) =xe^{\lambda_l x}$. Again, one can easily verify \eqref{eq:K1} by plugging $y_1 (c_i h)= c_i h {\rm e}^{c_i \lambda_l h }, \ y_1 (0)=0, \ y'_1 (0)=1, $ and $y''_1 (c_j h)=(2\lambda_l +\lambda^2_l  c_j h) {\rm e}^{ c_j \lambda_l h }$ into \eqref{eq:LiLm0}.
Similarly, \eqref{eq:K2} is confirmed for $K \ge 2$ (i.e., $m=0, 1, 2, \ldots, K$) by considering $y_K (x)=x^K{\rm e}^{ \lambda_l x}$.
\end{proof}

As a direct consequence of Theorem~\ref{theorem3.1}, we obtain the following result  for EFTDDIRK methods of degree $(0, L)$.
\begin{mycor}\label{corollary3.1}
An EFTDDIRK method of degree $(0, L)$ with the coefficients $b_i (\mu)$ expanded as \eqref{Coeffh:b}
has order of consistency at least three.
\end{mycor}
\begin{proof}
Based on Theorem~\ref{theorem3.1}, we have that an EFTDDIRK method of degree $(0, L)$ satisfies the fitting condition \eqref{eq:K0}.
Employing the Taylor expansions  ${\rm e}^{c_i \lambda_l h}=1+ c_i \lambda_l h + \mathcal{O}(h^2)$ and ${\rm e}^{\lambda_l h}=1+ \lambda_l h+ \frac{1}{2}(\lambda_l h)^2+\frac{1}{6} (\lambda_l h)^3+ \mathcal{O}(h^4)$, one can express  \eqref{eq:K0b} as
\begin{equation} \label{eq:29}
\sum_{i=1}^{s} b_i (\mu) +\big(\sum_{i=1}^{s} b_i (\mu)c_i \big)\lambda_l h + \mathcal{O}(h^2) =  \frac{1}{2} +\frac{1}{6} \lambda_l h+ \mathcal{O}(h^2).
\end{equation}
Inserting $b_i(\mu)=b_i^{(0)} + \mathcal{O}(\omega^2 h^2)$ (from \eqref{Coeffh})
into \eqref{eq:29} and taking the limit of both sides as $h$ approaches 0 shows that
\begin{equation} \label{eq:30}
\sum_{i=1}^{s} b^{(0)}_i =  \frac{1}{2}.
\end{equation}
With this, \eqref{eq:29} can be simplified as
\begin{equation} \label{eq:31}
\big(\sum_{i=1}^{s} b^{(0)}_i c_i - \frac{1}{6} \big)\lambda_l   + \mathcal{O}(h) =   \mathcal{O}(h),
\end{equation}
which implies
\begin{equation} \label{eq:32}
 \sum_{i=1}^{s} b^{(0)}_i c_i =  \frac{1}{6}
\end{equation}
 when $h$ approaches 0.
Clearly,   \eqref{eq:30} and  \eqref{eq:32} are the order conditions for EFTDDIRK methods of order 3.
\end{proof}
We note that this result is similar to a result for the trigonometrically fitted two-derivative Runge--Kutta methods \cite{FangYouMing2014} which holds for $s \ge 2$.

Next, in order to allow a direct treatment of oscillatory solutions, we now consider the case
$\lambda_l=\pm {\rm i}l\omega$, $l=0, 1, \ldots, L$.
Since $\mu={\rm i} \omega h$, we have $\lambda_l h= \pm l \mu$.
Therefore, the fitting condition \eqref{eq:K0} of EFTDDIRK methods of degree $(0, L)$ becomes
\begin{subequations} \label{eq:33}
\begin{align}
 &  (l \mu)^2 \sum_{j=1}^{i} a_{ij} (\mu) {\rm e}^{\pm c_j  l \mu} \pm \xi_i (\mu) c_i  l \mu= {\rm e}^{\pm c_i  l \mu} -1,  \label{eq:33a}\\
 &(l \mu)^2 \sum_{i=1}^{s} b_i (\mu) {\rm e}^{\pm c_i  l \mu}  =  {\rm e}^{ \pm l \mu} -1  \mp l \mu \label{eq:33b}.
 \end{align}
\end{subequations}
When $l=1$, we have the following observation.
\begin{mycor}\label{corollary3.2}
The fitting conditions in \eqref{eq:33} for an EFTDDIRK method of degree $(0, 1)$ whose coefficients expanded as in \eqref{Coeffh}  imply the following
\begin{subequations} \label{eq:add}
\begin{align}
&  \xi^{(0)}_i =1, \  \sum_{j=1}^{i} a^{(0)}_{ij} =\frac{c^2_i}{2}, \label{eq:add:a}\\
& \sum_{j=1}^{i} a^{(0)}_{ij} c_j +  \xi^{(2)}_i c_i =   \frac{1}{3!} c^3_i,  \label{eq:add:a1}\\
&\frac{1}{2!}  \sum_{j=1}^{i} a^{(0)}_{ij} c^2_j +  \sum_{j=1}^{i} a^{(2)}_{ij}=   \frac{1}{4!} c^4_i,  \label{eq:add:a2}\\
& \sum_{i=1}^{s} b^{(0)}_i =  \frac{1}{2!}, \   \sum_{i=1}^{s} b^{(0)}_i c_i =  \frac{1}{3!} , \label{eq:add:b}\\
& \frac{1}{2!} \sum_{i=1}^{s} b^{(0)}_i c^2_i +\sum_{i=1}^{s} b^{(2)}_i = \frac{1}{4!} , \label{eq:add:c}\\
& \frac{1}{3!} \sum_{i=1}^{s} b^{(0)}_i c^3_i +\sum_{i=1}^{s} b^{(2)}_i c_i= \frac{1}{5!} , \label{eq:add:d}\\
& \frac{1}{4!} \sum_{i=1}^{s} b^{(0)}_i c^4_i + \frac{1}{2!}\sum_{i=1}^{s} b^{(2)}_i c^2_i +\sum_{i=1}^{s} b^{(4)}_i= \frac{1}{6!}. \label{eq:add:e}
 \end{align}
\end{subequations}
\end{mycor}
 \begin{proof}
 The proof is straightforward by inserting \eqref{Coeffh} and the Taylor series expansions of $e^{c_j \mu}$, $e^{c_i \mu}$, $e^{\mu}$ into   \eqref{eq:33} with $l=1$ and comparing term by term on both sides of each condition. We omit the details. Here, we note that \eqref{eq:add:b} is already obtained in Corollary~\ref{corollary3.1} for the more general case, and \eqref{eq:add:a} justifies the simplifying assumption needed in Theorem~\ref{ClassicalOrderConditions}.
\end{proof}

Finally, we justify the assumption on the coefficients $a_{ij} (\mu), b_i (\mu),$ and $ \xi_i (\mu)$ stated in the Subsection \ref{subsec2.1}.
\begin{mylem}\label{lemma3.2}
If  the fitting conditions in \eqref{eq:33} are held for all $\mu$, the coefficients $ a_{ij} (\mu), b_i (\mu),$ and $ \xi_i (\mu)$ must be even functions.
\end{mylem}
\begin{proof}
Adding the two equations in \eqref{eq:33a} gives
\begin{equation} \label{eq:34}
(l \mu)^2 \sum_{j=1}^{i} a_{ij} (\mu)\big( {\rm e}^{c_j  l \mu} + {\rm e}^{-c_j  l \mu} \big) = {\rm e}^{c_i  l \mu}+ {\rm e}^{-c_j  l \mu}  -2.
\end{equation}
Interchanging $\mu \rightarrow -\mu$ and subtracting the resulted equation from  \eqref{eq:34} leads to
\begin{equation} \label{eq:35}
\sum_{j=1}^{i} [a_{ij} (\mu)- a_{ij} (-\mu) ] \big( {\rm e}^{c_j  l \mu} + {\rm e}^{-c_j  l \mu} \big) = 0.
\end{equation}
This shows $a_{ij} (\mu) = a_{ij} (-\mu)$ due to the fact that the functions $\{ {\rm e}^{c_j  l \mu} + {\rm e}^{-c_j  l \mu} \}_{j=1}^{i}$ are linearly independent.
Similarly, using  \eqref{eq:33b} one can show that  $b_i (\mu) = b_i (-\mu)$.
Next, interchanging $\mu \rightarrow -\mu$ for the first equation in \eqref{eq:33a} and subtracting from the other, we have
\begin{equation} \label{eq:36}
(l \mu)^2 \sum_{j=1}^{i} [a_{ij} (\mu)- a_{ij} (-\mu) ] {\rm e}^{-c_j  l \mu} -[\xi_i (\mu)-\xi_i (-\mu)]c_i \mu = 0.
\end{equation}
Since $a_{ij} (\mu) = a_{ij} (-\mu)$, one derives $\xi_i (\mu)=\xi_i (-\mu)$.
\end{proof}

\section{Construction of EFTDDIRK methods}\label{Construct}
In this section, using the results presented in Theorem~\ref{ClassicalOrderConditions} and Theorem~\ref{theorem3.1}, we construct  EFTDDIRK methods based on the reference set \eqref{Fae} with $K=0$ and $L=1$, i.e.,  methods of degree $(0, 1)$. This is the case in which these methods can integrate exactly differential equations with oscillating  solutions involving
$ {\rm e}^{\pm{\rm i} \omega x}=\cos( \omega x) \pm {\rm i} \sin(\omega x) $.
%
 Since EFTDDIRK methods are at least of order 3, i.e., the order conditions \eqref{d1}--\eqref{d2} are automatically satisfied (as shown in Corollary~\ref{corollary3.1}), we will derive  methods of orders 4, 5, and 6 by using the fitting conditions in \eqref{eq:33} and the remaining required order conditions \eqref{OrderFour}--\eqref{OrderSix}.

 Clearly, with $s=1$, it is not possible to construct fourth-order  EFTDDIRK methods. Therefore, we start off our construction with $s=2$.
For later display the coefficients of our EFTDDIRK methods in a compact form, we denote
\begin{equation}\label{E_terms}
  {\rm E_\mu}^{+}(\zeta) = {\rm e}^{\zeta \mu}+{\rm e}^{-\zeta \mu}, \qquad {\rm E_\mu}^{-}(\zeta) = {\rm e}^{\zeta \mu}-{\rm e}^{-\zeta \mu}, \quad \zeta \in \mathbb{R}.
\end{equation}
Clearly, given $\zeta$ and $\mu$, one can compute these terms (involving the sum and difference of exponential terms) directly (e.g., using the available MATLAB function $\mathtt{exp}$) without truncating their Taylor series expansions.

\begin{myrem}
Since $\mu={\rm i} \omega h$, using the Euler's formula one can also represent the terms ${\rm E_\mu}^{+}(\zeta)$ and ${\rm E_\mu}^{-}(\zeta)$ in \eqref{E_terms} as
\begin{subequations} \label{sincos}
\begin{align}
  {\rm E_\mu}^{+}(\zeta) & = {\rm e}^{{\rm i} \zeta \omega h}+ {\rm e}^{-{\rm i} \zeta \omega h} = 2 \cos(\zeta \omega h),\\
   {\rm E_\mu}^{-}(\zeta) & =  {\rm e}^{{\rm i} \zeta \omega h}- {\rm e}^{-{\rm i} \zeta \omega h}  = 2 {\rm i}  \sin(\zeta \omega h).
 \end{align}
\end{subequations}
Therefore, we note that while the coefficients of all our newly constructed EFTDDIRK methods below in this section are displayed in terms of $  {\rm E_\mu}^{+}(\cdot)$ and ${\rm E_\mu}^{-}(\cdot)$, they actually  involve $\sin( \zeta \omega h)$ and  $\cos( \zeta \omega h)$ (with appropriate constants  $\zeta \in \mathbb{R}$ depending on each method).
\end{myrem}

\subsection{Two--stage fourth-order methods}\label{2Stage4OrderConstruct}
The fitting conditions in \eqref{eq:33} and the required order conditions \eqref{OrderFour} for this case ($s=2$) now read as
\begin{subequations} \label{2ExpFittingCond}
\begin{align}
& \mu^2 a_{11}(\mu){\rm e}^{{\pm} c_1\mu} \pm \xi_1(\mu)c_1 \mu={\rm e}^{\pm c_1\mu}- 1,  \label{2ExpFittingConda}\\[2ex]
& \mu^2a_{21}(\mu){\rm e}^{{\pm} c_1\mu}+\mu^2a_{22}(\mu){\rm e}^{{\pm} c_2\mu}\pm \xi_2(\mu)c_2 \mu={\rm e}^{\pm c_2\mu}-1, \label{2ExpFittingCondb}\\[2ex]
& \mu^2
b_{1}(\mu){\rm e}^{{\pm} c_1 \mu}+\mu^2
b_{2}(\mu){\rm e}^{{\pm} c_2 \mu}={\rm e}^{\pm \mu}-1 \mp \mu. \label{2ExpFittingCondc}\\
& b^{(0)}_1 c^2_1 + b^{(0)}_2 c^2_2 =\frac{1}{12},  \label{4ths2:a} \\
& b^{(2)}_1  + b^{(2)}_2  =0,   \label{4ths2:b}
 \end{align}
\end{subequations}
respectively.
While solving \eqref{2ExpFittingConda} gives $a_{11}(\mu)$ and $ \xi_1(\mu)$ at once, solving \eqref{2ExpFittingCondc} gives  $b_{1}(\mu)$ and $b_{2}(\mu)$.
Since \eqref{2ExpFittingCondb} includes two equations with three unknown coefficients, one can take one of them as a free parameter. For instance, we take
 $a_{21}(\mu)$ as a free parameter and set it as $a_{21}(\mu)=\phi$.
Putting altogether, we display the solution to \eqref{2ExpFittingCond} as follows:
\begin{equation}\label{coeffEFTDDIRK2s4}
\begin{array}{l}
  a_{11}(\mu)=\tfrac{1}{\mu^2} (1-\tfrac{2}{{\rm E_\mu}^{+}(c_1)}), \ a_{21}(\mu)=\phi, \
  a_{22}(\mu)=\tfrac{{\rm E_\mu}^{+}(c_2)-(2+\phi \mu^2{\rm E_\mu}^{+}(c_1))}{\mu^2{\rm E_\mu}^{+}(c_2)},\\[2ex]
\xi_1(\mu)=\tfrac{{\rm E_\mu}^{-}(c_1)}{c_1\mu {\rm E_\mu}^{+}(c_1)}, \quad
\xi_2(\mu)=\tfrac{{\rm E_\mu}^{-}(c_2)-\phi \mu^2 {\rm E_\mu}^{-}(c_1-c_2)}{c_2\mu {\rm E_\mu}^{+}(c_2)}
, \\[2ex]
  b_1(\mu)=\tfrac{{\rm E_\mu}^{-}(c_2)+{\rm E_\mu}^{-}(1-c_2)-\mu {\rm E_\mu}^{+}(c_2)}{\mu^2 {\rm E_\mu}^{-}(c_1-c_2)}, \
  b_2(\mu)=\tfrac{\mu {\rm E_\mu}^{+}(c_1)-{\rm E_\mu}^{-}(c_1)-{\rm E_\mu}^{-}(1-c_1)}{\mu^2 {\rm E_\mu}^{-}(c_1-c_2)}.
\end{array}
\end{equation}
Next, we solve for the two order conditions \eqref{4ths2:a}-- \eqref{4ths2:b}.
Due to \eqref{eq:add:c} (see Corollary~\ref{corollary3.2} for $s=2$), we see that one only needs to satisfy one of them (as the other one will be then automatically satisfied).
For  instance, we solve \eqref{4ths2:a} by expanding $b_1(\mu)$ and $b_2(\mu)$ in  \eqref{coeffEFTDDIRK2s4} (with note that $\mu={\rm i} \omega h$) in Taylor series as
\begin{subequations}\label{Taylor_bi}
\begin{align}
  &b_1(\mu)=\tfrac{1-3c_2}{6(c_1-c_2)}+\tfrac{10(c_1-2c_2)(c_1-3c_1c_2)-20c_2^2+15c_2-3}{360(c_1-c_2)}\omega^2 h^2+\mathcal{O}(h^4) \label{b1mu}\\
  &b_2(\mu)=\tfrac{3c_1-1}{6(c_1-c_2)}+\tfrac{10(c_2-2c_1)(3c_1c_2-c_2)+20c_1^2-15c_1+3}{360(c_1-c_2)}\omega^2 h^2+\mathcal{O}(h^4)\label{b2mu}
\end{align}
\end{subequations}
(to get $b^{(0)}_1, b^{(0)}_2$), and thus obtain a constraint for $c_1$ and $c_2$:
\begin{equation}\label{Ord4Algebraic}
\dfrac{1-3c_2}{6(c_1-c_2)}c^2_1 + \dfrac{3c_1-1}{6(c_1-c_2)} c^2_2= \dfrac{1}{12}
\iff 2(c_1+c_2-3c_1c_2)-1=0
\end{equation}
for all $c_1 \ne c_2$.
Overall, this results in a family of fourth-order 2-stage methods which will be called
$\mathtt{EFTDDIRK2s4(c_1,c_2,\phi)}$.
For example, solving \eqref{Ord4Algebraic} with a choice of $c_1=1/4$ leads to $c_2=1$, denoted $\mathtt{EFTDDIRK2s4(\tfrac{1}{4},1,\phi)}$. Another solution is to choose $c_1=0$, resulting in $c_2=1/2$, and  $a_{11} (\mu)=0$ (the first stage is explicit), denoted $\mathtt{EFTDDIRK2s4(0,\tfrac{1}{2},\phi)}$.
The parameter $\phi$ will be determined by the optimizing the phase property of the methods. This will be discussed in the next section.

\subsection{Two--stage fifth-order methods}

In this subsection, we consider whether using $s=2$ is possible to derive  a  fifth-order method. For this, in addition to  \eqref{2ExpFittingCond}, the conditions  in \eqref{OrderFive} are required. Supposed that \eqref{d6} is satisfied, one derives $\bm{b}^{(2)} \cdot \bm{c}=0$ due to  \eqref{eq:add:c}  in Corollary~\ref{corollary3.2}. With this,  \eqref{d7} is now simplified to
\begin{equation}\label{eq:simplified}
\bm{b}^{(0)} \cdot (\bm{\xi}^{(2)} \text{*} \bm{c})=0 \iff b^{(0)}_1 \xi^{(2)}_1 c_1+  b^{(0)}_2  \xi^{(2)}_2 c_2 = 0.
\end{equation}
Next, using \eqref{eq:add:d} which can be written as $\bm{A}^{(0)}\bm{c} + \bm{\xi}^{(2)} \text{*} \bm{c}=  \tfrac{1}{3!}\bm{c}^3$, we have
$\bm{b}^{(0)}\cdot \big(\bm{A}^{(0)}\bm{c}\big)=\frac{1}{3!} \bm{b}^{(0)}\cdot \bm{c}^3 - \bm{b}^{(0)} \cdot (\bm{\xi}^{(2)} \text{*} \bm{c})=\tfrac{1}{3!} \frac{1}{20} - 0 = \tfrac{1}{120}$. This shows that
 \eqref{d5} is then automatically satisfied. Therefore, to fulfill  \eqref{OrderFive}, we eventually need to solve  \eqref{d6} and \eqref{eq:simplified} only. \\
Expanding $\xi_i(\mu)$ (see  \eqref{coeffEFTDDIRK2s4}) in Taylor series
\begin{equation}\label{Taylor_xi}
  \xi_1(\mu)=1+\tfrac{c_1^2}{3}\omega^2 h^2+\mathcal{O}(h^4), \
\xi_2(\mu)=1-\big(\phi-\tfrac{\phi c_1}{c_2}-\tfrac{c_2^2}{3} \big)\omega^2 h^2+\mathcal{O}(h^4)
\end{equation}
 to get $\xi^{(2)}_1, \ \xi^{(2)}_2$ and employing  \eqref{Taylor_bi}, the two conditions \eqref{d6} and \eqref{eq:simplified}  become
\begin{equation}\label{eq57}
\frac{1-3c_2}{6(c_1-c_2)}c^3_1 + \frac{3c_1-1}{6(c_1-c_2)} c^3_2= \frac{1}{20} \iff  c^2_1 +c^2_2+c_1c_2 -3c_1 c_2 (c_1 + c_2)=\frac{3}{10}
\end{equation}
(for all $c_1 \ne c_2$) and
$
\tfrac{1-3c_2}{6(c_1-c_2)}(\tfrac{c^3_1}{3}) + \tfrac{3c_1-1}{6(c_1-c_2)} (\phi-\frac{\phi c_1}{c_2}-\tfrac{c_2^2}{3})c_2  = 0,
$
respectively. Note that, with $c_1 \ne c_2$, the later equation can be simplified as
\begin{equation}\label{eq59}
c^2_1 +c^2_2 -3c_1 c_2 (c_1 + c_2) + c_1 c_2 + 3\phi (3c_1 -1)=0 \iff \phi=\frac{1}{10(1-3c_1)}
\end{equation}
by employing  \eqref{eq57}.
Clearly, $c_1$ and $c_2$ can be easily solved from the system of two algebraic equations  \eqref{Ord4Algebraic} (to fulfill  \eqref{2ExpFittingCond}) and  \eqref{eq57} (indeed, given the form of this system, $c_1$ and $c_2$ are the two roots of the quadratic equation
($10X^2-8X+ 1=0$). Then inserting them into \eqref{eq59} gives $\phi$. We display the results as follows
\begin{equation}\label{CoeffOrder5}
  c_1=\frac{1}{10}(4-\sqrt{6}), \quad c_2=\frac{1}{10}(4+\sqrt{6}), \quad \phi=\frac{1}{50}(2+3\sqrt{6}). \\
\end{equation}
This results in a 2-stage fifth-order method with the coefficients given in \eqref{coeffEFTDDIRK2s4} and \eqref{CoeffOrder5} which will be called
$\mathtt{EFTDDIRK2s5}$.

\subsection{Three--stage sixth-order method}
For a 3--stage method, the exponential fitting conditions using \eqref{eq:33} for this case ($s=3$) now gives
\begin{subequations}\label{3StageExpCond}
\begin{align}
&\mu^2a_{11}(\mu){\rm e}^{{\pm} c_1\mu}\pm \xi_1(\mu)c_1 \mu={\rm e}^{\pm c_1\mu}- 1 ,\label{3ExpFittingConda}\\
&\mu^2a_{21}(\mu){\rm e}^{{\pm} c_1\mu}+\mu^2a_{22}(\mu){\rm e}^{{\pm} c_2\mu}\pm \xi_2(\mu)c_2 \mu={\rm e}^{\pm c_2\mu}-1 ,\label{3ExpFittingCondb} \\
&\mu^2a_{31}(\mu){\rm e}^{{\pm} c_1\mu}+\mu^2a_{32}(\mu){\rm e}^{{\pm}
c_2\mu}+\mu^2a_{33}(\mu){\rm e}^{{\pm} c_3\mu}\pm \xi_3(\mu)c_3 \mu={\rm e}^{\pm c_3\mu}-1 ,\label{3ExpFittingCondc} \\
&\mu^2
b_{1}(\mu){\rm e}^{{\pm} c_1 \mu}+\mu^2
b_{2}(\mu){\rm e}^{{\pm} c_2 \mu}+\mu^2b_{3}(\mu){\rm e}^{{\pm} c_3 \mu}={\rm e}^{\pm \mu}-1 \mp \mu. \label{3ExpFittingCondd}
\end{align}
\end{subequations}
In addition to \eqref{3StageExpCond}, we require the coefficients to satisfy the classical order condition \eqref{d1}--\eqref{OrderSix}. As in the two-stage method, one can similarly solve \eqref{3ExpFittingConda} for $a_{11}(\mu)$ and $\xi_1(\mu)$. Next, we solve  parameters \eqref{3ExpFittingCondb} for $\xi_2(\mu)$ and $a_{22}(\mu)$, while we make $a_{21}(\mu)=\chi$ a free parameter. Furthermore, we solve \eqref{3ExpFittingCondc} for $\xi_3(\mu)$ and $a_{33}(\mu)$, while setting $a_{31}(\mu)=\beta$ and $a_{32}(\mu)=\delta$ as free parameters. Lastly, setting $b_2(\mu)=\eta$ as a free parameter in \eqref{3ExpFittingCondd}, we solve  for $b_1(\mu)$ and $b_3(\mu)$. The solution to \eqref{3StageExpCond} therefore yields the following:
\begin{equation*}
\begin{array}{l}
  a_{11}(\mu)=\tfrac{1}{\mu^2}\left(1-\tfrac{2}{{\rm E_\mu}^{+}(c_1)} \right), a_{21}(\mu)=\chi, \
  a_{22}(\mu)=\tfrac{{\rm E_\mu}^{+}(c_2)-(2+\chi \mu^2{\rm E_\mu}^{+}(c_1))}{\mu^2{\rm E_\mu}^{+}(c_2)},
\end{array} \hspace{2.5in}
\end{equation*}
\begin{equation*}
\begin{array}{l}
a_{33}(\mu)=\tfrac{{\rm E_\mu}^{+}(c_3)-\mu^2(\beta {\rm E_\mu}^{+}(c_1)+\delta {\rm E_\mu}^{+}(c_2))-2}{\mu^2 {\rm E_\mu}^{+}(c_3)}, \qquad \xi_1(\mu)=\tfrac{{\rm E_\mu}^{-}(c_1)}{c_1\mu {\rm E_\mu}^{+}(c_1)},
\end{array}\hspace{2in}
\end{equation*}
\begin{equation*}
\begin{array}{l}
   \xi_2(\mu)=\tfrac{{\rm E_\mu}^{-}(c_2)-\chi \mu^2 {\rm E_\mu}^{-}(c_1-c_2)}{c_2\mu {\rm E_\mu}^{+}(c_2)}, \qquad \xi_3(\mu)= \tfrac{{\rm E_\mu}^{-}(c_3)-\mu^2(\beta {\rm E_\mu}^{-}(c_1-c_3)+\delta {\rm E_\mu}^{-}(c_2-c_3))}{c_3\mu {\rm E_\mu}^{+}(c_3)},
\end{array}\hspace{2in}
\end{equation*}
\begin{equation*}
\begin{array}{l}
  b_1(\mu)=\tfrac{{\rm E_\mu}^{-}(c_3)+{\rm E_\mu}^{-}(1-c_3)-\mu {\rm E_\mu}^{+}(c_3)-\eta \mu^2 {\rm E_\mu}^{-}(c_2-c_3)}{\mu^2 {\rm E_\mu}^{-}(c_1-c_3)},
\end{array}\hspace{2in}
\end{equation*}
\begin{equation*}
\begin{array}{l}
  b_3(\mu)=\tfrac{\mu {\rm E_\mu}^{+}(c_1)-{\rm E_\mu}^{-}(c_1)-{\rm E_\mu}^{-}(1-c_1)-\eta \mu^2 {\rm E_\mu}^{-}(c_1-c_2)}{\mu^2 {\rm E_\mu}^{-}(c_1-c_3)}.
\end{array}\hspace{2in}
\end{equation*}
Now that we have the solution, we obtain the Taylors expansion of the coefficients and seek the free parameters in order to satisfy classical sixth-order
conditions \eqref{d1}--\eqref{OrderSix}. With the help of  Corollary~\ref{corollary3.2}, the classical order conditions \eqref{d3}, \eqref{d6}, \eqref{d7}, \eqref{d8}, \eqref{d9}, and \eqref{d10} are sufficient to attain order six. These set of conditions, yield a system of cumbersome algebraic equations, which are omitted here. The
free parameters satisfy the sixth-order conditions with
$$
\begin{array}{c}
  c_1=0, \qquad c_2=\dfrac{1}{10}(5-\sqrt{5}), \qquad
c_3=\dfrac{1}{10}(5+\sqrt{5}), \qquad
\chi=\dfrac{1}{30}(3-\sqrt{5}), \\[2ex]
\beta=\dfrac{1}{60}(1+\sqrt{5}), \qquad
\delta=\dfrac{1}{60}(5+3\sqrt{5}), \qquad
\eta=\dfrac{1}{24}(5+\sqrt{5}).
\end{array}
$$
This method is denoted as $\mathtt{EFTDDIRK3s6}$.

\begin{myrem}
If the frequency $\omega$ of the problem is close to 0 (so does  $\mu = {\rm i} \omega h$), for practical computation,  it is then preferable to compute the coefficients of our EFTDDIRK methods based on their truncated Taylors series. We note, however, that this is not the case for our numerical examples presented in Section~\ref{experiments}.
\end{myrem}

\section{Phase and stability properties}\label{StabilityPhase}
This section is concerned with the stability and phase-lag analysis
of the EFTDDIRK methods derived in Section \ref{Construct}. Following \cite{FangYouMing2014,Vyver2005} for oscillatory systems, we
apply the method \eqref{modTDDIRK} to the test equation
\begin{equation}\label{testeq}y'={\rm i}\Lambda y, \quad
{\rm i}^2=-1, \quad  \Lambda>0.
\end{equation}
This results in the following difference equation
\begin{equation}\label{recursivey}
y_{n+1}=R(\theta,\omega h)y_n, \quad \theta=\Lambda h,
\end{equation}
where $R(\theta,\omega h)$ is the imaginary stability function of $\theta$ and $\omega h$ given as
\begin{equation}\label{Stabfun}
   R(\theta,\omega h)=\big(1-\theta^2 \bm{b}(\mu)^T(I_s+\theta^2\bm{A}(\mu))^{-1}\bm{e}) \big)+{\rm i} \big(\theta(1-\theta^2)\bm{b}(\mu)^T(I_s+\theta^2\bm{A}(\mu))^{-1}(\bm{\xi}*\bm{c})
    \big)
\end{equation}
(here, $\mu={\rm i}\omega h$, $I_s$ is the $s \times s$ identity matrix).
\subsection{Phase properties}
The dispersion and dissipation are important properties which characterize the numerical behavior of methods constructed for
oscillatory problems.  Similarly to \cite{Vyver2005,FangYouMing2014},  they can be defined for our proposed EFTDDIRK methods as follows.

\begin{mydef}[Dispersion and dissipation]
With the stability function $R(\theta,\omega h)$ given in~\eqref{Stabfun}, the quantities
\begin{equation}\label{dispdis}
{\rm Disp}(\theta)=\theta-{\rm arg}(R(\theta,\omega h)) \ \text{and} \ \
{\rm Dis}(\theta)=1-| R(\theta,\omega h)|
\end{equation}
are called the dispersion (phase-lag) and the
dissipation (amplification error), respectively.
The scheme \eqref{modTDDIRK} is dispersive of order $p$ and is dissipative of order $q$ if
$$
{\rm Disp}(\theta)=C_{p+1}(r)\theta^{p+1}+\mathcal{O}(\theta^{p+3}), \quad {\rm Dis}(\theta)=C_{q+1}(r)\theta^{q+1}+\mathcal{O}(\theta^{q+3}),
$$
respectively (here $r=\frac{\omega h}{\theta}$).  In the case ${\rm Disp}(\theta)= 0$ or ${\rm Dis}(\theta)=0$, it is called zero-dispersive or
zero-dissipative, respectively.
\end{mydef}
Using \eqref{dispdis}, in Table \ref{PLETable} we derive the dispersion and dissipation for the EFTDDIRK methods constructed in Section~\ref{Construct}.
\setlength{\extrarowheight}{3 pt}
\renewcommand{\arraystretch}{1.5}
\begin{table}[H]\small
  \centering
  \begin{tabular}{|l|l|}
    \hline
    Method  & Dispersion ${\rm Disp}(\theta)$ \\
            & Dissipation ${\rm Dis}(\theta)$\\
    \hline
    $\mathtt{EFTDDIRK2s4(\tfrac{1}{4},1,\phi)}$  & $\tfrac{1}{480}(-11+20\phi)(1-r^2)\theta^5+ \mathcal{O}(\theta^7)$\\
 & $\tfrac{1}{23040}(-230+360\phi-7r^2)(1-r^2)\theta^6+ \mathcal{O}(\theta^8)$
\\
\hline
    $\mathtt{EFTDDIRK2s4(0,\tfrac{1}{2},\phi)}$  &
$\tfrac{1}{240}(-3+40\phi)(1-r^2)\theta^5+ \mathcal{O}(\theta^7)$\\
                 & $\tfrac{1}{5760}(-50+720\phi-r^2)(1-r^2)\theta^6+ \mathcal{O}(\theta^8)$  \\
    \hline
    $\mathtt{EFTDDIRK2s5}$ & $\tfrac{(1-r^2) \big((168 \sqrt{6}-379) r^2+84 \sqrt{6}-162\big)}{252000}\theta^7+ \mathcal{O}(\theta^9)$\\
 & $\tfrac{ \left(r^4+r^2-2\right)}{14400}\theta^6+ \mathcal{O}(\theta^8)$
\\
    \hline
    $\mathtt{EFTDDIRK3s6}$ & $\tfrac{(1-r^2) \big((17 \sqrt{5}-10) r^2-4 \sqrt{5}-10\big)}{3780 (5+\sqrt{5})^3}\theta^7+ \mathcal{O}(\theta^9)$\\
& $\tfrac{ (1-r^2) \big((15+\sqrt{5})
r^4+175 (5 \sqrt{5}-9) r^2-175 (1+3
\sqrt{5})\big)}{15120000 (3+\sqrt{5})}\theta^8+ \mathcal{O}(\theta^{10})$\\
    \hline
  \end{tabular}
 \caption{Dispersion and dissipation of the newly derived EFTDDIRK methods.}\label{PLETable}
\end{table}

In view of Table \ref{PLETable}, it is easy to see that by choosing
$\phi=\tfrac{11}{20}$ and $\phi=\tfrac{3}{40}$, the phase-lag for $\mathtt{EFTDDIRK2s4(\tfrac{1}{4},1,\phi)}$ and
$\mathtt{EFTDDIRK2s4(0,\tfrac{1}{2},\phi)}$ is optimized and increased to order six, respectively. In Section~\ref{experiments}, we demonstrate the efficiency of these optimized methods over non-optimized phase-lag methods (which we simply take $\phi=0$).

\subsection{Region of imaginary stability}
One can also study the imaginary stability region of  the proposed  EFTDDIRK methods similarly to \cite{Vyver2005,FangYouMing2014}.
\begin{mydef}[Imaginary stability region]
The region of imaginary stability $\mathcal{S}$ of  the EFTDDIRK methods  \eqref{modTDDIRK}  is given by
$$
\mathcal{S}=\{(\theta, \omega h)\ |\  \theta>0, \omega >0, \
|R(\theta,\omega h)|\leq1\}.
\vspace{-8pt}
$$
\end{mydef}
\vspace{-8pt}
\begin{figure}[ht!]
\centering
$$
\begin{array}{ccc}
  \mathtt{EFTDDIRK2s4(\tfrac{1}{4},1,0)} & \mathtt{EFTDDIRK2s4(\tfrac{1}{4},1,\tfrac{11}{20})} & \mathtt{EFTDDIRK2s4(0,\tfrac{1}{2},0)} \\
  \includegraphics[width=4cm,height=3.5cm]{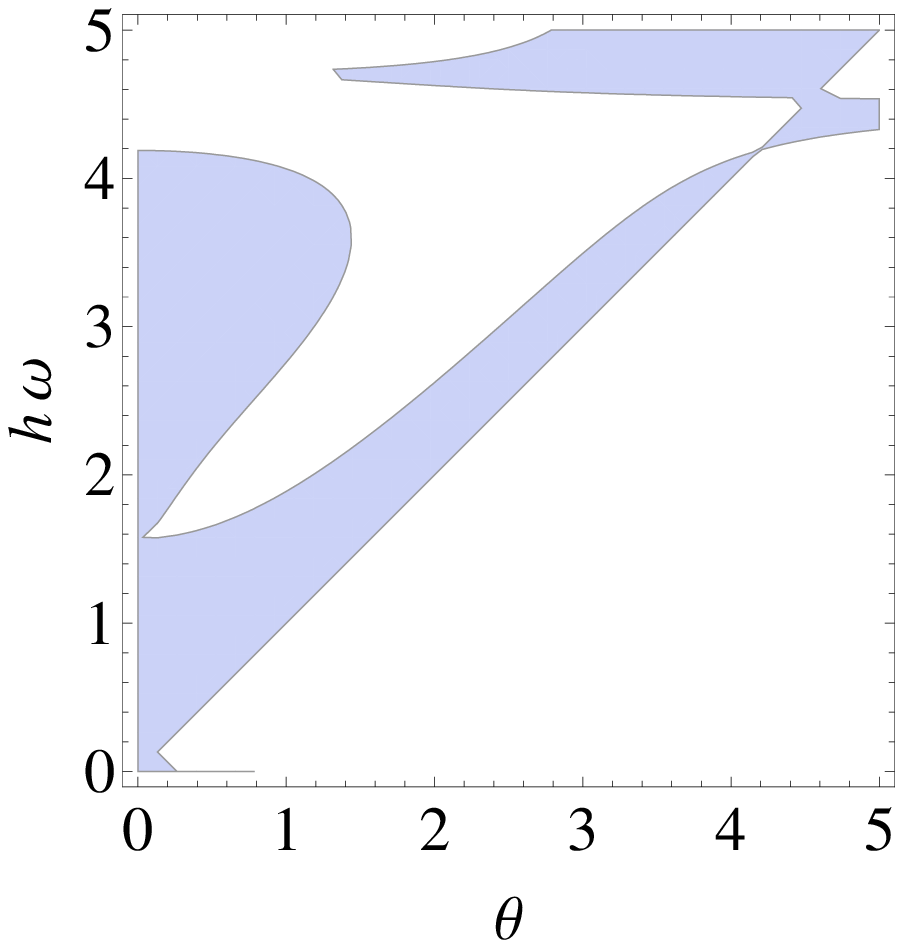} & \includegraphics[width=4cm,height=3.5cm]{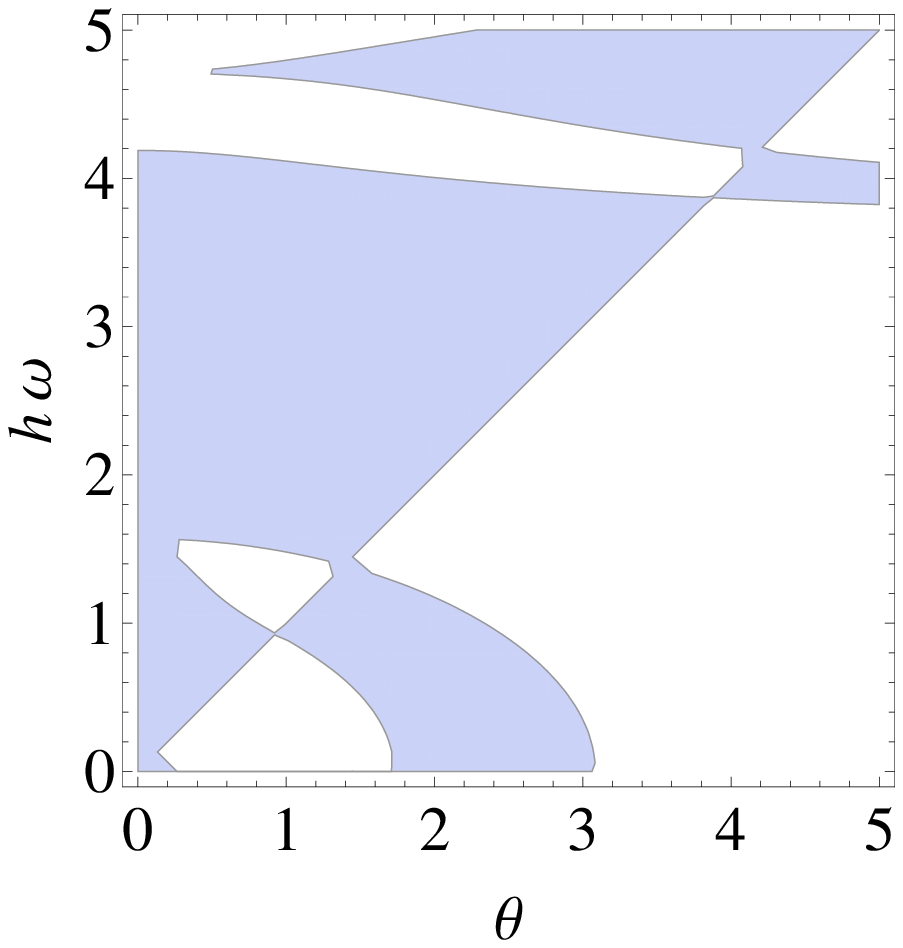} & \includegraphics[width=4cm,height=3.5cm]{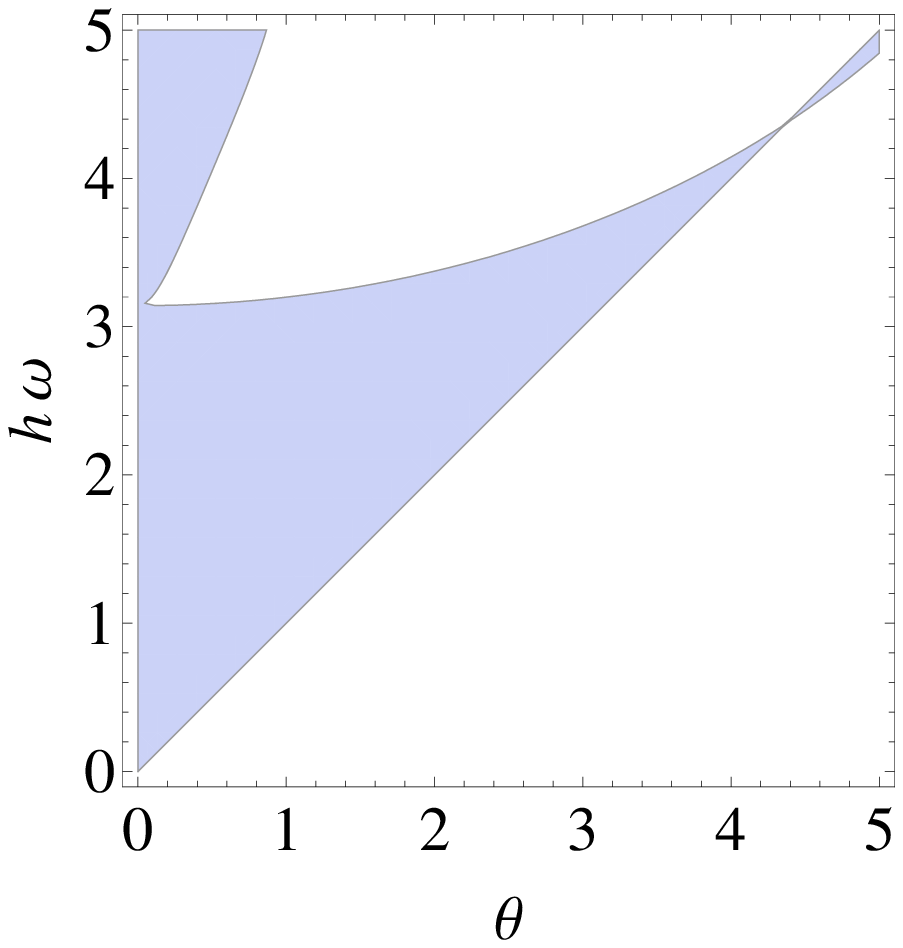} \\
  \mathtt{EFTDDIRK2s4(0,\tfrac{1}{2},\tfrac{3}{40})} & \mathtt{EFTDDIRK2s5} & \mathtt{EFTDDIRK3s6}\\
  \includegraphics[width=4cm,height=3.5cm]{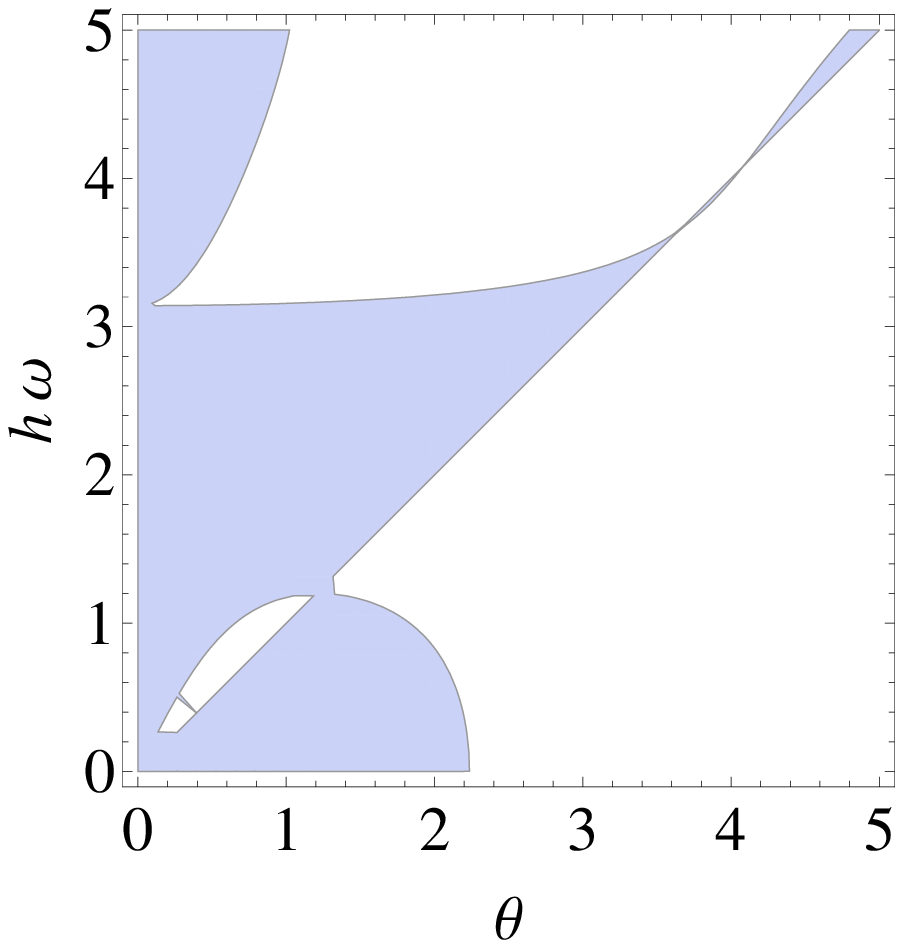} & \includegraphics[width=4cm,height=3.5cm]{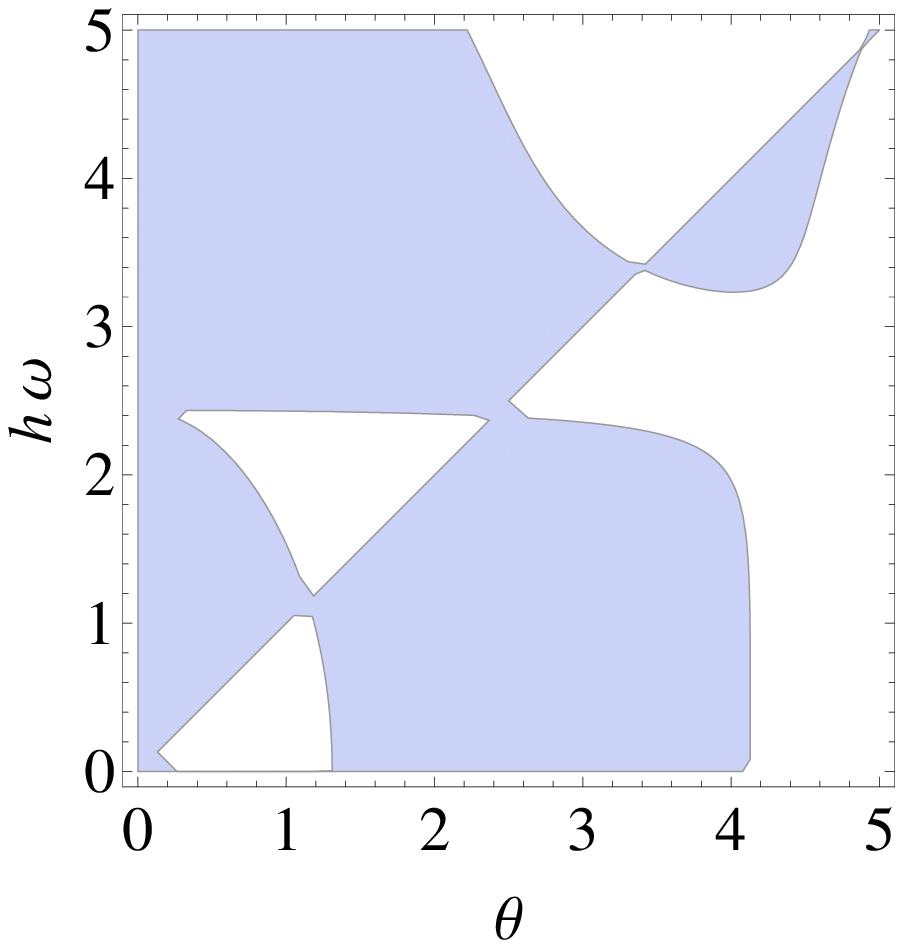} & \includegraphics[width=4cm,height=3.5cm]{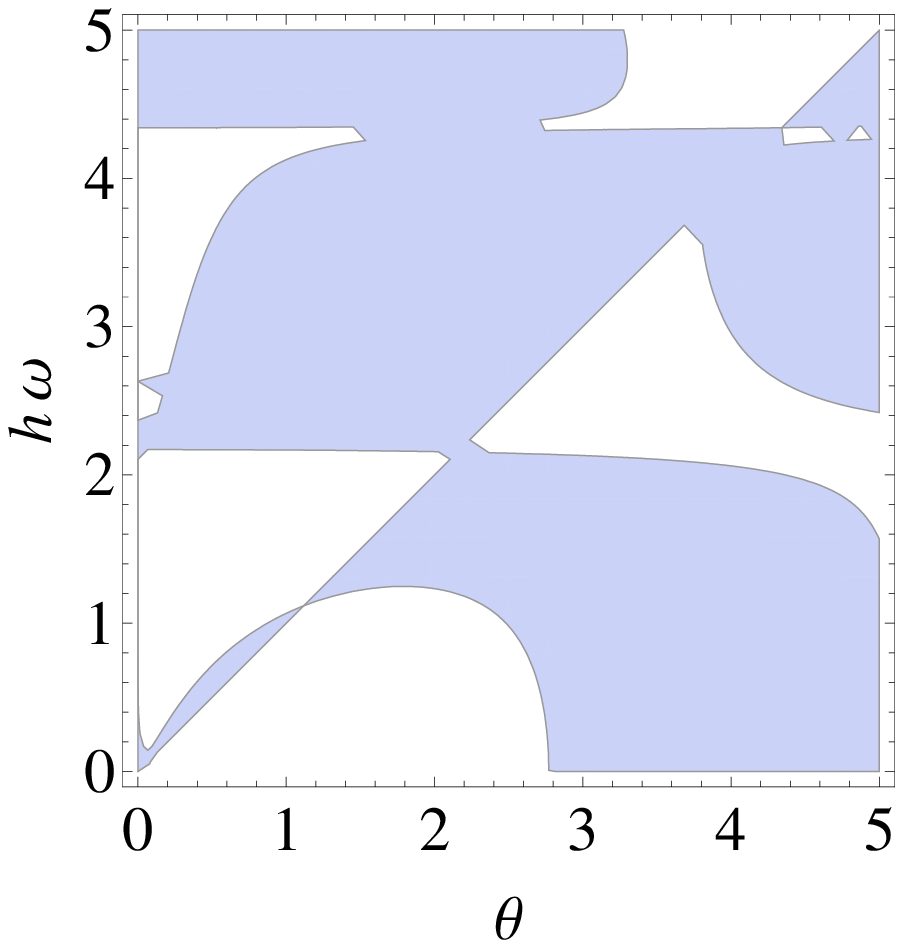}
  \vspace{-8pt}
\end{array}
$$
\vspace{-8pt}
\caption{Imaginary stability plots for the newly derived EFTDDIRK methods}\label{FigStabEFTDDIRK}
\end{figure}
In Figure~\ref{FigStabEFTDDIRK}, we plot the imaginary stability regions in the $\theta-\omega h$ plane on $[0, 5]^2$ (see the shaded regions) of the newly constructed EFTDDIRK methods. For a fixed value of $\omega h$, one can determine a sequence of the imaginary stability intervals for each EFTDDIRK method (by finding the intersection of the horizontal line passing through $\omega h$ crossing the shaded region). It is observed from Figure \ref{FigStabEFTDDIRK} that the optimized methods $\mathtt{EFTDDIRK2s4(\tfrac{1}{4},1,\tfrac{11}{20})}$ and $\mathtt{EFTDDIRK2s4(0,\tfrac{1}{2},\tfrac{3}{40})}$ have larger imaginary stability regions than their non-optimized phase-lag counterparts $\mathtt{EFTDDIRK2s4(\tfrac{1}{4},1,0)}$ and $\mathtt{EFTDDIRK2s4(0,\tfrac{1}{2},0)}$, respectively. The sixth-order method $\mathtt{EFTDDIRK3s6}$ has the largest stability region. In fact, it has the largest sequence of imaginary stability intervals even for larger values of $\omega h$.

\section{A note on frequency estimation}\label{FreqDetermine}

In view of the constructed EFTDDIRK schemes, it is crucial to determine the principal frequency $\omega$ (in turn $\mu={\rm i}\omega h$) for their implementation. This was a challenging aspect of the numerical integration of initial value problems with
exponentially/trigonometrically based methods, especially when the frequency is not known in advance. In \cite{IxaruBergheMeyer2002}, a strategy was derived for estimating the frequency of the system based on the leading term of the local truncation error. This approach has been extended and explored in \cite{Vyver2005}. Another approach was discussed in \cite{RamosVigo2010,VigoRamos2015} to obtain the optimum frequency ($\omega_{opt}$) as a result of minimizing the total energy of nonlinear periodic oscillators.
In this work, we apply the strategy presented in \cite{VigoRamos2015} for the problem where the fitting frequency is not given. In particular, the \emph{golden section search} technique \cite{Press2007} is utilized to obtain the optimum frequency ($\omega_{opt}$) based on minimizing the
error of the method for a given interval around the angular frequency.

\section{Numerical experiments}\label{experiments}

In this section, we evaluate the effectiveness of the newly constructed EFTDDIRK methods of orders 4, 5, and 6 when compared to existing implicit methods of the same orders in the literature. Our numerical experiments are carried out on a list of three oscillatory test problems (see below) and implementations are performed in MATLAB on a single workstation using a 8GB RAM processor Intel(R) Core(TM) i5-8250U CPU @ 1.80GHz Laptop.  Numerical investigation include accuracy and efficiency comparisons. For accuracy comparisons, all methods use the same set of stepsizes. However, for efficiency comparisons,  
 the stepsizes are chosen such that all the considered methods achieve the same error thresholds (measured based on the maximum global error ($\log_{10}(MGE)$)). When the exact solution is unknown, the reference solution is computed by using the sixth-order method $\mathtt{EFTDDIRK3s6}$ with sufficient small stepsize.

\subsection*{Computation of the internal stages.}

The internal stages  $Y_i\approx y(x_n+c_ih)$, $i=1,2,\ldots, s$ of our EFTDDIRK methods are sequentially computed by using the  fixed point iteration technique, which is given as
\begin{equation}\label{InternalStageComputation}
\begin{array}{l}
Y_i^{(0)}=y_n+hc_if_n+\dfrac{(c_ih)^2}{2}g_n,\\
Y_i^{(r+1)}=y_n+h\xi_i(\mu)c_if_n+h^2\Big(\sum\limits_{j=1}^{i-1}a_{ij}(\mu)f(Y_j)+a_{ii}(\mu)f(Y_i^{(r)})\Big).
\end{array}
\end{equation}
The stoping criterion for the iterative procedure \eqref{InternalStageComputation} is
$$
\|Y_i^{(r+1)} -Y_i^{(r)} \|_2 < tol = 10^{-12}, \quad r=0,1,2,\ldots
$$
where $Y_i^{(r)}$ is the value at the $r$th iteration in the iterative process.

For the readers' convenience, we list the evaluated methods in two groups as follows.

\medskip
\noindent {\bf Methods of Order 4}
\begin{itemize}
    \item $\mathtt{TDFIRK2s4}$: 2-stage two-derivative implicit RK method \cite{ChanWangTsai2012}.
    \item $\mathtt{TDDIRK2s4}$: 2-stage two-derivative DIRK method \cite{AhmadSenuIbrahim2019}.
    \item $\mathtt{EFSSDIRK3s4}$ 3-stage symmetric and symplectic DIRK method \cite{EhigieDiaoZhangFangHouYou2017}.
    \item $\mathtt{TFTDDIRK2s4}$ 2-stage EFTDDIRK method of order 4 \cite{AhmadSenuIbrahimOthman2019}
    \item $\mathtt{EFTDDIRK2s4(c_1,c_2,\phi)}$: new 2-stage EFTDDIRK method of order 4.
\end{itemize}
\medskip
{\bf Methods of Order 5 and 6}
\begin{itemize}

    \item $\mathtt{TDFIRK3s5}$: 3-stage implicit two-derivative RK method of order 5 \cite{ChanWangTsai2012}.
    \item $\mathtt{Gauss3}$: Gauss 3-stage implicit RK method of order 6  \cite{Butcher2016}.
    \item $\mathtt{TDDIRK3s5}$: 3-stage two-derivative DIRK method of
    order 5 \cite{AhmadSenuIbrahim2019}.
    \item $\mathtt{EFTDDIRK2s5}$: new 2-stage EFTDDIRK method of order 5.
    \item $\mathtt{TDDIRK4s6}$: 4-stage two-derivative DIRK method of
    order 6 \cite{AhmadSenuIbrahim2019}.
    \item $\mathtt{EFSSIRK3s6a}$: new 3-stage symmetric and symplectic implicit Runge-Kutta method of order 6 \cite{CalvoFrancoMontijanoRandez2008}.
    \item $\mathtt{EFSSIRK3s6b}$: new 3-stage symmetric and symplectic implicit Runge-Kutta method of order 6 \cite{CalvoFrancoMontijanoRandez2009}
    \item $\mathtt{EFTDDIRK3s6}$: new 3-stage EFTDDIRK method of order 6.
\end{itemize}

\medskip
\medskip
\noindent {\bf Example~1 (perturbed Kepler's problem).}
Consider the Hamiltonian system studied in \cite{FrancoGomez2014}
\begin{equation}\label{PbH}
    H(p,q)=\dfrac{1}{2}(p_1^2+p_2^2)+\dfrac{\omega^2}{2}(q_1^2+q_2^2)+\dfrac{\alpha}{6}(q_1^2+q_2^2)^3,
\end{equation}
with the initial data
$$
q_1(0)=1, \quad q_2(0)=0, \quad p_1(0)=0, \quad
p_2(0)=\omega+\epsilon,
$$
where $\alpha=\epsilon(2\omega+\epsilon)$. The analytic solution
is given by
$$
\begin{array}{l}
  q_1(t)=\cos((\omega+\epsilon)t), \quad
p_1(t)=-(\omega+\epsilon)\sin((\omega+\epsilon)t),\\
  q_2(t)=\sin((\omega+\epsilon)t), \quad \
p_2(t)=(\omega+\epsilon)\cos((\omega+\epsilon)t),
\end{array}
$$
which presents oscillations (see Figure~\ref{FigProb1Time}). In this experiment, we have chosen the parameter
values $\epsilon=10^{-2}$, $\omega=5$, and the integration is
carried out on the interval $[0,100]$. For accuracy comparisons, the same set of stepsizes
$\{h=\frac{1}{2^i}$, $j=3,4,5,6\}$ is used for all the considered integrators. The numerical results
are presented in Figures \ref{FigProb14} and \ref{FigProb156}, which also show the efficiency plots (the step sizes are chosen in such a way that the same error thresholds are achieved).
As seen from the left diagrams of these figures, all the new methods fully achieve their orders of convergence (4, 5, and 6).
\begin{figure}[H]
  \centering
  \includegraphics[width=9cm]{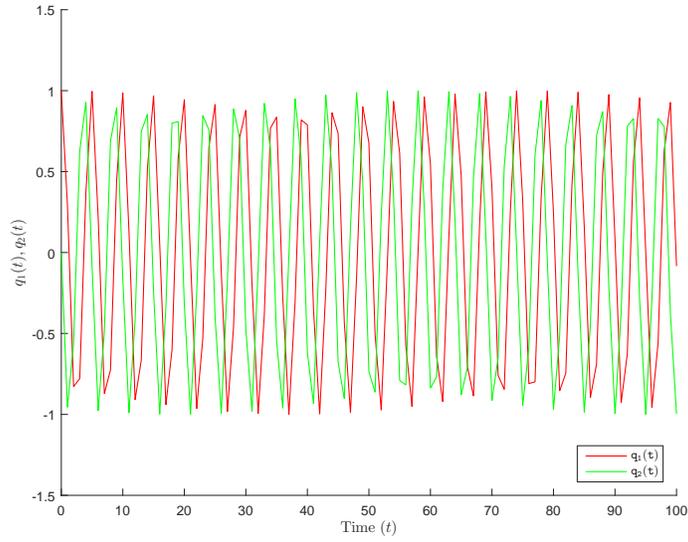}\\
  \caption{Solution components $q_1(t)$ and $q_2(t)$ of Example~1 showing oscillatory behavior.}
  \label{FigProb1Time}
\end{figure}

\begin{figure}[H]
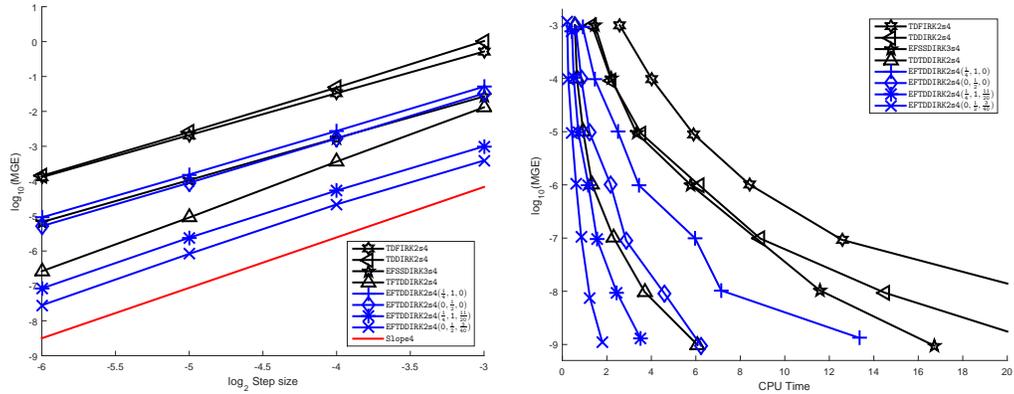

\centering
\begin{tabular}{cc}
\epsfig{file=NLProb1Acc4.eps,width=0.47\linewidth,clip=}
\hspace{2mm}
\epsfig{file=NLProb1CPU4.eps,width=0.47\linewidth,clip=}
\end{tabular}
\caption{Accuracy (left) and efficiency (right) plots of 4th-order methods for Example~1. For comparison, a straight line with slope 4 is added.}\label{FigProb14}
\end{figure}
\begin{figure}[H]
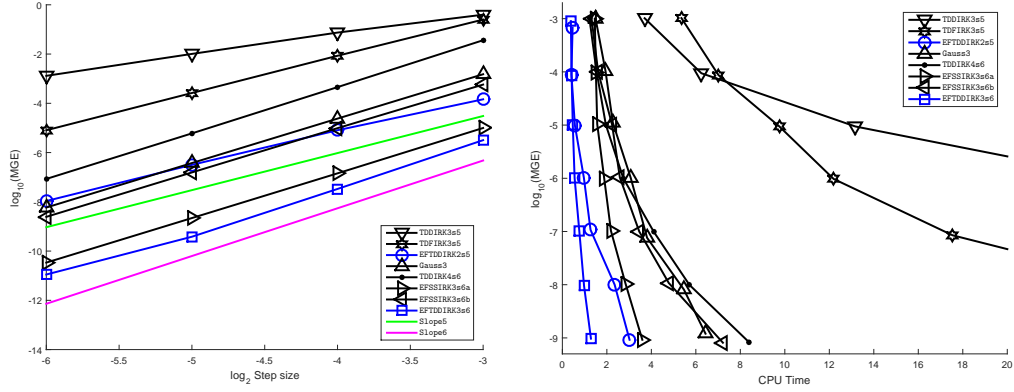

\centering
\begin{tabular}{cc}
\epsfig{file=NLProb1Acc56.eps,width=0.47\linewidth,clip=}
\hspace{2mm}
\epsfig{file=NLProb1CPU56.eps,width=0.47\linewidth,clip=}
\end{tabular}
\caption{Accuracy (left) and efficiency (right) plots of 5th and 6th-order methods for Example~1. For comparison, straight lines with slopes 5 and 6 are added.}\label{FigProb156}
\end{figure}

It can be seen from Figure \ref{FigProb14} that the methods whose dispersion were optimized
($\mathtt{EFTDDIRK2s4(\tfrac{1}{4},1,\tfrac{11}{20})}$ and
$\mathtt{EFTDDIRK2s4(0,\tfrac{1}{2},\tfrac{3}{40})}$) clearly outperform the other
methods of order 4.
Figure \ref{FigProb156} shows (right) that the newly derived methods $\mathtt{EFTDDIRK2s5}$
(order 5) and $\mathtt{EFTDDIRK3s6}$ (order 6) derived in this work are much more accurate and faster when compared to the considered existing methods of order 5 and 6. \\

\noindent \textbf{Example~2 (Fermi-Pasta-Ulam problem)}.
Next, we consider the highly oscillatory Fermi-Pasta-Ulam (FPU) problem (see \cite{HairerLubichWanner2006})  including $m$ stiff springs, in which  the motion is described by a second-order system of differential equations of the form
\begin{equation}\label{FermiPasta}
  \ddot{x}(t)+\Omega^2 x(t)=-\nabla U(x(t)), t\in [t_0,t_{end}],
\end{equation}
where
\begin{equation*}
\Omega=\left[ \begin{array}{ccc}
  {\bf 0}_{m \times m} & {\bf 0}_{m \times m} \\
  {\bf 0}_{m \times m} & \omega {\bf I}_{m \times m}
\end{array} \right]
\qquad  (\text{with} \ \omega \gg 1),
\end{equation*}
and $U(x)$ is a smooth nonlinear potential function given by
\begin{equation*}
U(x)=\frac{1}{4}\big[(x_{1}-x_{m+1})^4+ \sum_{j=1}^{m-1} (x_{j+1}-x_{m+j-1}-x_{j}-x_{m+j})^4+(x_{m}+x_{2m})^4  \big]
\end{equation*}
with $x_{j}=x_j (t)$ represents for positions of the $j$th stiff spring.
As in \cite{HairerLubichWanner2006}, we consider the case for $m=3$,
and choose
$$
x_1(0)=1, \quad \dot{x}_1(0)=1, \quad x_4(0)=\omega^{-1}, \quad
\dot{x}_4(0)=1,
$$
and zero for the remaining initial values. The
system is integrated on $[0,100]$ with $\omega=50$.

Note that the FPU problem \eqref{FermiPasta} can be also
 described by the Hamiltonian system with total energy
 \begin{equation}\label{HamFerm}
 H(x,\dot{x})=\dfrac{1}{2}\|\dot{x}\|^2+\dfrac{1}{2}\|\Omega x\|^2+U(x),
\end{equation}
where $x$ and $\dot{x}$ expresses the scaled displacements and velocities (or
momenta), respectively. Therefore, the exact value of the total energy is
$H(x_0,\dot{x}_0)$\\$=\tfrac{1}{2}(\sqrt{2})^2+\tfrac{1}{2}(1^2)+(0.501200080)=2.001200080$.

While the left diagrams of Figures \ref{FigProb24} and \ref{FigProb256} show the accuracy comparisons (using the same set of step sizes $\{h=1/2^j$, $j=6,7,8,9\}$ for each method),
the right diagrams display  the efficiency evaluations, in which the step sizes are chosen in such a way that the same error thresholds are achieved.

\begin{figure}[H]
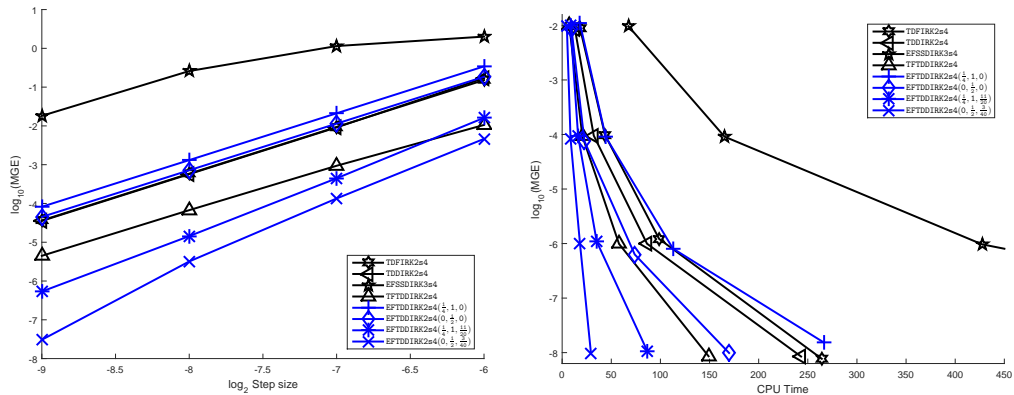

\centering
\begin{tabular}{cc}
\epsfig{file=NLProb2Acc4.eps,width=0.47\linewidth,clip=}
\hspace{2mm}
\epsfig{file=NLProb2CPU4.eps,width=0.47\linewidth,clip=}
\end{tabular}
\caption{Accuracy (left) and efficiency (right) plots of 4th-order methods for Example~2.}\label{FigProb24}
\end{figure}
\begin{figure}[H]
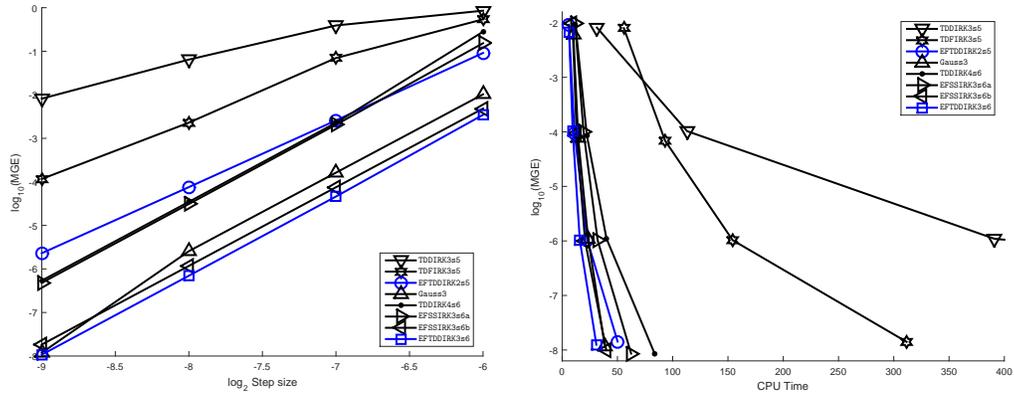

\centering
\begin{tabular}{cc}
\epsfig{file=NLProb2Acc56.eps,width=0.47\linewidth,clip=}
\hspace{2mm}
\epsfig{file=NLProb2CPU56.eps,width=0.47\linewidth,clip=}
\end{tabular}
\caption{Accuracy (left) and efficiency (right) plots of 5th and 6th-order methods for Example~2.}\label{FigProb256}
\end{figure}
Among methods of order 4, one can see again that the new method
$\mathtt{EFTDDIRK2s4(0,\tfrac{1}{2},\tfrac{3}{40})}$ is the most accurate and efficient.
 Also, the newly derived fifth- and sixth- order methods $\mathtt{EFTDDIRK2s5}$
 and $\mathtt{EFTDDIRK3s6}$ are more accurate and efficient compared to some existing methods of the same order, respectively.

Next, we investigate the preservation of the Hamiltonian for the FPU system by some selected methods of orders 4, 5, and 6. Figure~\ref{HamiltonianError} presents the absolute error of the Hamiltonian ($|H_N-H_0|$) versus time using stepsize $h=\tfrac{1}{200}$, where $H_N$ is the computed Hamiltonian after $N$ steps.
\begin{figure}[H]
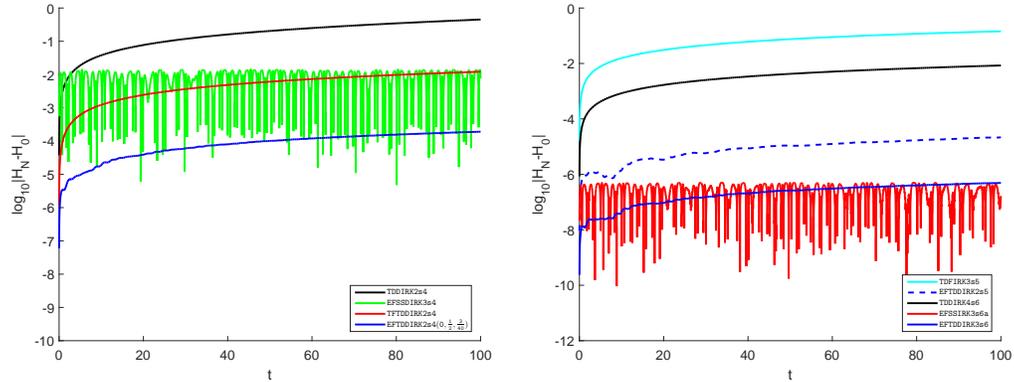

\centering
\begin{tabular}{cc}
\epsfig{file=NLProb2EvoErrHam4.eps,width=0.47\linewidth,clip=}
\hspace{2mm}
\epsfig{file=NLProb2EvoErrHam56.eps,width=0.47\linewidth,clip=}
\end{tabular}
\caption{Hamiltonian errors of 4th-order (left), 5th- and 6th-order (right) methods for Example~2.}\label{HamiltonianError}
\end{figure}
It is observed that the most accurate methods of the new derived fourth-order methods $\mathtt{EFTDDIRK2s4(0,\tfrac{1}{2},\tfrac{3}{40}})$ preserves the Hamiltonian best. However, the newly derived fifth- and sixth-order methods $\mathtt{EFTDDIRK2s5}$ and
 $\mathtt{EFTDDIRK3s6}$ preserve the total energy much better than the existing methods of the same order, respectively.
\newline
\newline
\noindent \textbf{Example~3 (Sine-Gordon equation)}.
We consider the sine-Gordon nonlinear  equation with periodic boundary
condition (see \cite{Weinberger1965})
\begin{equation}\label{eq:sinegordon}
\left\{ \begin{array}{l}
  \dfrac{\partial^2 u}{\partial t^2}=\dfrac{\partial^2
u}{\partial x^2}-\sin u, \qquad -1<x<1, \quad t>0 \\[2ex]
  u(-1,t)=u(1,t). \\
\end{array} \right.
\end{equation}
A semi-discretization in the spatial variable by the second-order centered finite difference method leads to the following system of ODEs
\begin{equation}\label{DiscretizedSineGordon}
    \dfrac{d^2U}{dt^2}+M U=F(U), \qquad 0<t \leq t_{end},
\end{equation}
where $U(t)=(u_1(t),\ldots u_N(t))^T$ with $u_i (t)=u(x_i,t)$,
$i=1,\ldots,N$. Equation \eqref{DiscretizedSineGordon} can be
further transformed to a system of first order DEs given by
\begin{equation}\label{FirstOrderDiscretizedSineGordon}
    \dfrac{d}{dt}\left[%
\begin{array}{c}
  U \\
  V \\
\end{array}%
\right]=\left[%
\begin{array}{cc}
  {\bf 0} & I \\
  -M & {\bf 0} \\
\end{array}%
\right] \left(%
\begin{array}{c}
  U \\
  V \\
\end{array}%
\right)+\left[%
\begin{array}{c}
   \bar{0} \\
  F(U) \\
\end{array}%
\right], \qquad 0<t \leq t_{end}.
\end{equation}
Here, $V=U'$, $I$ is the $N \times N$ identity matrix, {\bf 0} is the $N \times N$
zero matrix, $\bar{0}$ is a zero column vector of size $N \times 1$,
$$
M=\dfrac{1}{\Delta x^2}\left(%
\begin{array}{ccccc}
  2 & -1 &  &  & -1 \\
  -1 & 2 & -1 &  &  \\
   & \ddots & \ddots & \ddots &  \\
   &  & -1 & 2 & -1 \\
  -1 &  &  & -1 & 2 \\
\end{array}%
\right)
$$
(with $\Delta x=2/N$, $x_i=-1+i \Delta x$, $i=1,2,\ldots,N$), and
$F(U)=-\sin(U)=-(\sin u_1,\ldots, \sin u_N)T$. As in \cite{Franco2006}, we use the initial conditions

$$
U(0)=(\pi)_{i=1}^N, \qquad V(0)=\sqrt{N}\left(0.01+\sin
\left(\dfrac{2 \pi i}{N} \right) \right)_{i=1}^N,
$$
($N=64$) and integrate the problem on the interval $[0,10]$.\\
\begin{figure}[H]
\centering
\begin{tabular}{cc}
\epsfig{file=NLProb3Acc4.eps,width=0.47\linewidth,clip=}
\hspace{2mm}
\epsfig{file=NLProb3CPU4.eps,width=0.47\linewidth,clip=}
\end{tabular}
\caption{Accuracy (left) and efficiency (right) plots of 4th-order methods for Example~3.
}\label{FigProb34}
\end{figure}
\begin{figure}[H]
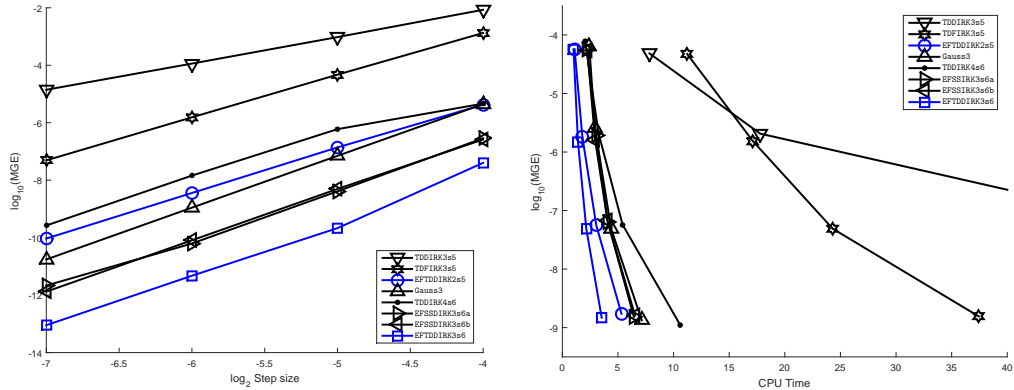

\centering
\begin{tabular}{cc}
\epsfig{file=NLProb3Acc56.eps,width=0.47\linewidth,clip=}
\hspace{2mm}
\epsfig{file=NLProb3CPU56.eps,width=0.47\linewidth,clip=}
\end{tabular}
\caption{Accuracy (left) and efficiency (right) plots of 5th and 6th-order methods for Example~3.}\label{FigProb356}
\end{figure}
Again, on the left side of Figure \ref{FigProb34} we display accuracy plots using  the same set of stepsizes $h=1/2^i$, $i=6,7,8,9$,
and the efficiency plots are shown on the right side (different time step sizes were chosen so that all the compared methods attain about the same level of accuracy). 

The numerical results show that $\mathtt{EFTDDIRK2s4(0,\tfrac{1}{2},0)}$ is the least accurate but the most efficient among the tested methods. This can be explained by the fact that its first stage is computed explicitly (since $c_1=0$).
 One can also see that the optimized fourth-order method $\mathtt{EFTDDIRK2s4(0, \tfrac{1}{2},\tfrac{3}{40})}$  performs the best overall compared to other fourth-order methods.

Next, we also investigate the spatial grid effect on this problem. For this, we compare the maximum global error obtained by all the methods for different values of $N$ (25, 50, and 100) with stepsize $h=\tfrac{1}{16}$ at time $t=10$. We present the obtained results in Tables \ref{TableSpatialGridEffect416} and \ref{TableSpatialGridEffect5616}.
From these results, we observe  that the accuracy decreases when increasing $N$. This is due to the stiffness of the problem which is increasing  for larger spatial grid  sizes $N$, which consequently affects to the global error.

\begin{table}[H]
  \centering
  \begin{tabular}{lccc}
    \hline
    Method & $N=25$ & $N=50$ & $N=100$ \\
    \hline
    $\mathtt{TDDIRK2s4}$ & $1.05 \cdot 10^{-4}$ & $1.00 \cdot 10^{-3}$ & $2.76 \cdot 10^{-2}$ \\
    $\mathtt{TDFIRK2s4}$ & $1.45 \cdot 10^{-4}$ & $1.60 \cdot 10^{-3}$ & $1.04 \cdot 10^{-2}$ \\
    $\mathtt{EFSSDIRK3s4}$ & $6.15 \cdot 10^{-5}$ & $9.18 \cdot 10^{-4}$ & $3.20 \cdot 10^{-3}$ \\
    $\mathtt{TFTDDIRK2s4}$ & $4.07 \cdot 10^{-7}$ & $1.34 \cdot 10^{-5}$ & $1.27 \cdot 10^{-5}$ \\
    $\mathtt{EFTDDIRK2s4(\tfrac{1}{4},1,0)}$ & $7.35 \cdot 10^{-7}$ & $4.29 \cdot 10^{-5}$ & $1.71 \cdot 10^{-4}$ \\
    $\mathtt{EFTDDIRK2s4(0,\tfrac{1}{2},0)}$ & $1.00 \cdot 10^{-3}$ & $7.69 \cdot 10^{-4}$ & $3.00 \cdot 10^{-3}$ \\
    $\mathtt{EFTDDIRK2s4(\tfrac{1}{4},1,\tfrac{11}{20})}$ & $1.64 \cdot 10^{-5}$ & $1.78 \cdot 10^{-4}$ & $5.58 \cdot 10^{-4}$ \\
    $\mathtt{EFTDDIRK2s4(0,\tfrac{1}{2},\tfrac{3}{40})}$ & $2.47 \cdot 10^{-7}$ & $3.09 \cdot 10^{-6}$ & $2.34 \cdot 10^{-5}$ \\
    \hline
  \end{tabular}
  \caption{Errors for 4th-order methods for different values of $N$ using $h=\tfrac{1}{16}$.}\label{TableSpatialGridEffect416}
\end{table}
\begin{table}[H]
  \centering
  \begin{tabular}{lccc}
    \hline
    Method & $N=25$ & $N=50$ & $N=100$ \\
    \hline
    $\mathtt{TDDIRK3s5}$ & $1.80 \cdot 10^{-4}$ & $1.40 \cdot 10^{-3}$ & $4.62 \cdot 10^{-2}$ \\
    $\mathtt{TDFIRK3s5}$ & $8.59 \cdot 10^{-6}$ & $8.29 \cdot 10^{-5}$ & $1.06 \cdot 10^{-2}$ \\
    $\mathtt{EFTDDIRK2s5}$ & $6.46 \cdot 10^{-8}$ & $5.83 \cdot 10^{-7}$ & $4.64 \cdot 10^{-5}$ \\
    $\mathtt{Gauss3}$ & $9.76 \cdot 10^{-8}$ & $1.92 \cdot 10^{-6}$ & $1.16 \cdot 10^{-5}$ \\
    $\mathtt{TDDIRK4s6}$ & $1.07 \cdot 10^{-6}$ & $2.21 \cdot 10^{-5}$ & $6.63 \cdot 10^{-4}$ \\
    $\mathtt{EFSSIRK3s6a}$ & $6.79 \cdot 10^{-9}$ & $4.48 \cdot 10^{-8}$ & $7.56 \cdot 10^{-7}$ \\
    $\mathtt{EFSSIRK3s6b}$ & $3.38 \cdot 10^{-9}$ & $1.07 \cdot 10^{-7}$ & $2.01 \cdot 10^{-6}$ \\
    $\mathtt{EFTDDIRK3s6}$ & $1.28 \cdot 10^{-9}$ & $6.79 \cdot 10^{-9}$ & $1.64 \cdot 10^{-7}$ \\
    \hline
  \end{tabular}
  \caption{Errors for 5th- and 6th-order methods for different values of $N$ using $h=\tfrac{1}{16}$.}\label{TableSpatialGridEffect5616}
\end{table}
\noindent \textbf{Example~4 (An ``almost" periodic orbit problem)}. Lastly, we consider the almost periodic orbit problem studied in \cite{SteifelBettis1969} given by
$$
y''+y=0.001 {\rm e}^{{\rm i}x}, \quad y(0)=1,\quad y'(0)=0.9995i,
$$
whose analytical solution is given by
$$
y(x)=(\cos x+0.0005x\sin x)+ {\rm i}(\sin x-0.0005 x\cos x),
$$
which represent a motion on a perturbation of a circular orbit in the complex plane. Clearly, for this problem $\omega=1$.

We integrate the problem on $[0, 1000]$ using all the methods listed above. Numerical results were obtained for all the methods using stepsizes $h=1/2^n$, $n=0,1,2,3$ and are illustrated in Figures \ref{FigProb44} and \ref{FigProb456}.

\begin{figure}[H]
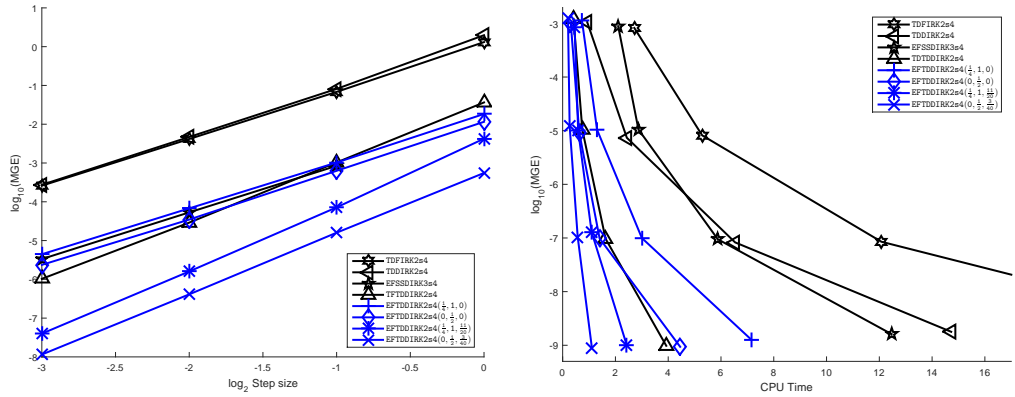

\centering
\begin{tabular}{cc}
\epsfig{file=NLProb4Acc4.eps,width=0.47\linewidth,clip=}
\hspace{2mm}
\epsfig{file=NLProb4CPU4.eps,width=0.47\linewidth,clip=}
\end{tabular}
\caption{Accuracy (left) and efficiency (right) plots of 4th-order methods for Example~4.}\label{FigProb44}
\end{figure}
\begin{figure}[H]
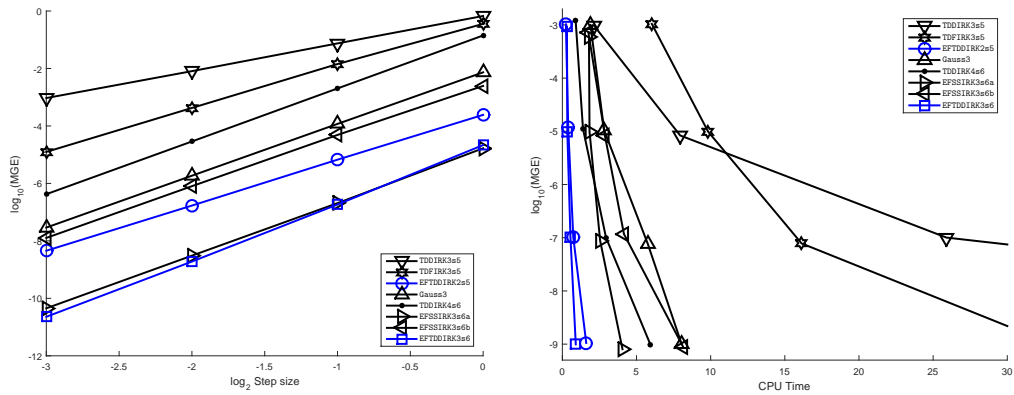

\centering
\begin{tabular}{cc}
\epsfig{file=NLProb4Acc56.eps,width=0.47\linewidth,clip=}
\hspace{2mm}
\epsfig{file=NLProb4CPU56.eps,width=0.47\linewidth,clip=}
\end{tabular}
\caption{Accuracy (left) and efficiency (right) plots of 5th and 6th-order methods for Example~4.}\label{FigProb456}
\end{figure}
In view of Figures \ref{FigProb44} and \ref{FigProb456}, the methods whose dispersion were optimized
($\mathtt{EFTDDIRK2s4(\tfrac{1}{4},1,\tfrac{11}{20})}$ and
$\mathtt{EFTDDIRK2s4(0,\tfrac{1}{2},\tfrac{3}{40})}$) clearly outperform the other
methods of order 4. Besides, we observe clearly from efficiency curves of Figures \ref{FigProb44} and \ref{FigProb456} that the newly derived methods $\mathtt{EFTDDIRK2s5}$
(order 5) and $\mathtt{EFTDDIRK3s6}$ (order 6) derived in this work are much more accurate and faster when compared to the considered existing methods. \\

\section{Conclusion}\label{conclu}

We have derived a class of exponentially fitted two-derivative diagonally implicit Runge--Kutta  (EFTDDIRK) methods for solving oscillatory differential equations.
New order and exponential fitting conditions are obtained, leading to the derivation of new methods of orders 4, 5, and 6. Stability and phase-lag analysis of these methods were investigated which resulted in optimized fourth-order schemes that are much more accurate and efficient.
Our numerical experiments have confirmed the efficiency and accuracy of  these new EFTDDIRK methods when compared to standard implicit (two-derivative) Runge--Kutta methods of the same orders.

Future works will be focusing on the existence of symmetric EFTDDIRK methods for symplectic, Hamiltonian or reversible systems. Also, we will consider the two-derivative singly diagonally implicit methods with a view to practical applications because of the easy implementation structure as discussed in \cite{Butcher2016} and \cite{KennedyCarpenter2016}.

\section*{Acknowledgements} J.O. Ehigie is grateful for the financial
support of EMS-Simons for Africa program, which enabled him undertake
some aspect of this research at Southern Methodist University, TX, USA.
V.T.~Luan gratefully acknowledges the support of the National Science Foundation under grant NSF DMS–2012022.
X. You would like to acknowledge the support of the National Natural Science Foundation of China (No. 11171155, No. 11871268) and  Natural Science Foundation of Jiangsu Province, China (No. BK20171370).

\end{document}